\documentclass{article}
\usepackage{amsmath}
\usepackage{amsfonts}
\usepackage{graphicx, subfigure}
\usepackage[pdftex]{hyperref}
\usepackage[dutch,english]{babel}
\usepackage{amsmath,amssymb}
\usepackage{amsthm}
\usepackage[usenames,dvipsnames]{color}
\usepackage{tikz}
\usepackage{accents}
\usepackage{xcolor}
\usepackage{hyperref}
\usepackage{mathtools}
\usepackage{multicol}
\usepackage{comment}

\usepackage{tikz}
\usetikzlibrary{shadows, arrows}

\newcommand{\la}{\lambda}
\newcommand{\ul}[1]{\underline{#1}}
\DeclareMathOperator{\spn}{span}
\DeclareMathOperator{\syn}{Syn}
\DeclareMathOperator{\Id}{Id}

\theoremstyle{plain}
\newtheorem{thr}{Theorem}[section]

\newtheorem{lem}[thr]{Lemma}

\theoremstyle{definition}
\newtheorem{defi}[thr]{Definition}
\newtheorem{ex}[thr]{Example}
\theoremstyle{remark}
\newtheorem{remk}{Remark}
\theoremstyle{remark}

\newcommand{\field}[1]{\mathbb{#1}}
\newcommand{\R}{\field{R}}
\newcommand{\N}{\field{N}}

\newcommand{\F}{\field{F}}

\definecolor{wred}{rgb}{0.7,0.18,0.12}
\definecolor{wgreen}{rgb}{0.1,0.53,0.37}

\numberwithin{equation}{section}

\def\barray{\begin{eqnarray*}}             \def\earray{\end{eqnarray*}}
\def\beq{\begin{equation}} \def\eeq{\end{equation}}

\title{Center manifolds of coupled cell networks}
\author{Eddie Nijholt\footnote{\mbox{Department of Mathematics, VU University Amsterdam, The Netherlands, \href{mailto:eddie.nijholt@gmail.com}{eddie.nijholt@gmail.com} }  }, Bob Rink\footnote{ \mbox{Department of Mathematics, VU University Amsterdam, The Netherlands, \href{mailto:b.w.rink@vu.nl}{b.w.rink@vu.nl} }} {} and Jan Sanders\footnote{ \mbox{Department of Mathematics, VU University Amsterdam, The Netherlands, \href{mailto:jan.sanders.a@gmail.com}{jan.sanders.a@gmail.com} }}}
\date{\today}

\excludecomment{percent}

\begin{document}

\maketitle

\begin{abstract}
Dynamical systems with a network structure can display anomalous bifurcations as a generic phenomenon. As an explanation for this it has been noted that homogeneous networks can be realized as quotient networks of so-called fundamental networks. The class of admissible vector fields for these fundamental networks is equal to the class of equivariant vector fields of the regular representation of a monoid. Using this insight, we set up a framework for center manifold reduction in fundamental networks and their quotients. We then use this machinery to explain the difference in generic bifurcations between three example networks with identical spectral properties and identical robust synchrony spaces.
\end{abstract}

\section{Introduction}\label{sec0}
 
Network dynamical systems play an important role in many of the sciences, with applications ranging from population dynamics to neuron networks and from electrical circuits to the world wide web. Although they have sparked an overwhelming amount of research, many questions on network systems remain open or have unsatisfactory answers. This is partly due to the fact that network structure, although perfectly well-defined, often seems to deny a versatile geometric description.  Most striking in this respect is the fact that a specific network structure is not usually preserved under coordinate changes. 

A remarkable phenomenon that distinguishes network dynamical systems from arbitrary dynamical systems is synchronisation \cite{pikovsky}. Synchronisation occurs when the agents of a network behave in unison. An example is the simultaneous firing of neurons. This paper is concerned with synchrony breaking. This is the phenomenon that less synchronous states may emerge from synchronous states as model parameters vary. It has been observed that synchrony breaking is often governed by very unusual bifurcation scenarios \cite{bifurcations, anto4, jeroen,  dias, elmhirst, krupa, synbreak2, curious, claire, leite, synbreak}. It is a major challenge to explain why this occurs, and to provide an efficient methodology for the computation of these bifurcations.

One successful method for the  analysis of synchrony in networks is based on the so-called \textit{groupoid formalism}. This formalism  was developed by Golubitsky and Stewart et al. \cite{golstew, torok, stewartnature, pivato}  and recently reformulated in the language of \textit{graph fibrations} by Deville and Lerman \cite{deville}. See also \cite{field} for results in a similar spirit. We shall work with a different framework though, that appears more suitable for bifurcation theory. In fact, we shall use the language of \textit{hidden symmetry} as described in \cite{fibr, RinkSanders3, CCN}. This is inspired by the fact that many of the characteristics of synchrony breaking, including the existence of robust invariant spaces, degenerate eigenvalues and anomalous generic bifurcations, are quite prevalent in the setting of ODEs with symmetry. More precisely, our formalism exploits the fact that every homogeneous coupled cell network can be realised as the quotient network of a so-called \textit{fundamental network}. This latter network admits a purely geometric characterisation, as its admissible vector fields are exactly the equivariant vector fields of the regular representation of a monoid (a monoid is a semigroup with unit). 

It is well-known that all the bounded solutions that emerge from a synchronous steady state through a synchrony breaking bifurcation, are contained in a so-called local center manifold. The main result of this paper is a center manifold theorem for homogeneous networks. It states that network structure can somehow be preserved under center manifold reduction. This means in particular that the dynamics on the center manifold is restricted by symmetry. Theorem \ref{stellingintro} will be formulated more precisely as theorems \ref{T5} and \ref{T7}. 

\begin{thr}\label{stellingintro}
Let $\Gamma$ be an admissible vector field for a fundamental network with symmetry monoid $\Sigma$ and let $x_0$ be a fully synchronous steady state of $\Gamma$. Then there exists a $\Sigma$-invariant local center manifold $\mathcal{M}^c$ for $\Gamma$ near $x_0$. 

The restriction $\Gamma|_{\mathcal{M}^c}$ to this local center manifold is $\Sigma$-equivariant and can therefore be interpreted as an admissible vector field  in an appropriate way.  

A center manifold of each quotient of the fundamental network is contained  in the  center manifold of the fundamental network as a robust synchrony space.
\end{thr}
\noindent Theorem \ref{stellingintro} is reminiscent of the well-known fact that the local center manifold of an ODE with a compact symmetry group can be assumed symmetric \cite{field3, field4, constrainedLS, perspective, golschaef2, vdbauw}. The proof of this  latter result strongly depends on the fact that every compact group has an invariant measure, and hence this proof does not apply to semigroups and monoids. Our proof of theorem \ref{stellingintro} shows that this technical problem can be overcome for fundamental networks.

The remainder of this paper is organised as follows. In Section \ref{sec1} we illustrate the impact of hidden network symmetry at the hand of three examples. In Section \ref{sec2} we introduce our general setup, and recall some basic theorems on which this paper builds. In Section \ref{sec3} we prove a center manifold theorem for fundamental networks. After this, Sections \ref{sec4} and \ref{sec5} are concerned with the symmetry and synchrony properties that are preserved under center manifold reduction. Finally, in Section \ref{sec6} we apply our general results to three examples.

\subsection*{Acknowledgement}
The authors would like to thank Andr\'e Vanderbauwhede for explaining several subtleties of center manifold reduction, and for an inspiring and encouraging  discussion on the subject. We also thank Florian Noethen for pointing out several mistakes in an early version of this manuscript.

\section{Three examples}\label{sec1}
 To illustrate the impact of hidden symmetry, let us consider the following three networks. They will be the leading examples of this paper.
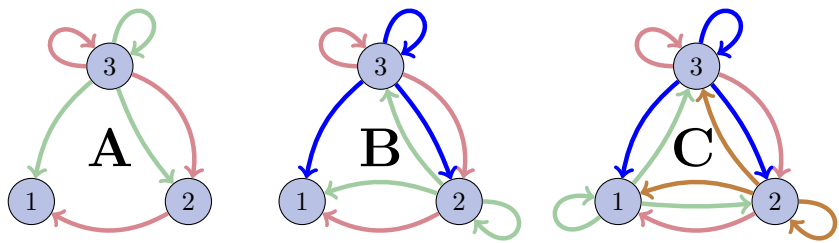
\begin{center}
\begin{tikzpicture}[scale=1.2]
\tikzstyle{every node} = [ draw, circle, fill={rgb:red,1;green,2;blue,6;white,20}, anchor=center]
\node (1) at ( -3.87, -0.5) {1};
\node (2) at ( -2.13, -0.5) {2};
\node (3) at (-3,1) {3};
\node (A) [draw=none, fill=none] at (-3,0.1) {\huge{\bf A}};
\path[ ->, line width=1.7pt,  color={rgb:red,12; blue,3; green, 2 ;white,15}]
	(2) edge[bend left=30] (1)
	(3) edge[bend left=40] (2)
	(3) edge[loop, in=130, out=180,  distance=7mm] (3);
\path[ ->, ,line width=1.7pt,  color={rgb:red,1; blue,1; green, 6 ;black,4;white,18}]
	(3) edge[bend right=20] (1)
	(3) edge[bend left=-10 ] (2)
	(3) edge[loop, in=30, out=80, min distance = 7mm ] (3);
\node (1) at ( -0.87, -0.5) {1};
\node (2) at ( 0.87, -0.5) {2};
\node (3) at (0,1) {3};
\node (B) [draw=none, fill=none] at (0,0.1) {\huge{\bf B}};
\path[ ->, line width=1.7pt,  color={rgb:red,12; blue,3; green, 2 ;white,15}]
	(2) edge[bend left=30] (1)
	(3) edge[bend left=40] (2)
	(3) edge[loop, in=130, out=180,  distance=7mm] (3);
\path[ ->, ,line width=1.7pt, color={rgb:red,1; blue,1; green, 6 ;black,4;white,18}]
	(2) edge[bend right=20] (1)
	(2) edge[loop, in=-50, out=0,  distance=7mm] (2)
	(2) edge[bend left= 15] (3);
\path[ ->, ,line width=1.7pt,  color={blue}]
	(3) edge[bend right=20] (1)
	(3) edge[bend left=10 ] (2)
	(3) edge[loop, in=30, out=80, min distance = 7mm ] (3);
\node (1) at ( 2.63, -0.5) {1};
\node (2) at ( 4.37, -0.5) {2};
\node (3) at (3.5,1) {3};
\node (C) [draw=none, fill=none] at (3.46,0.1) {\huge{\bf C}};
\path[ ->, line width=1.7pt,  color={rgb:red,12; blue,3; green, 2 ;white,15}]
	(2) edge[bend left=30] (1)
	(3) edge[bend left=40] (2)
	(3) edge[loop, in=130, out=180,  distance=7mm] (3);
\path[ ->, ,line width=1.7pt, color={rgb:red,1; blue,1; green, 6 ;black,4;white,18}]
	(1) edge[loop, distance=7mm, in = 170, out = 220] (1)
	(1) edge[bend left = -5  ] (2)
	(1) edge[bend left= -10] (3);
\path[ ->, ,line width=1.7pt,  color={blue}]
	(3) edge[bend right=20] (1)
	(3) edge[bend left=10 ] (2)
	(3) edge[loop, in=30, out=80, min distance = 7mm ] (3);
\path[ ->, ,line width=1.7pt,  color={brown}]
	(2) edge[bend right=20] (1)
	(2) edge[loop, in=-50, out=0,  distance=7mm] (2)
	(2) edge[bend left= 15] (3);
\end{tikzpicture}
\end{center}
 Networks {\bf A}, {\bf B} and {\bf C} give rise to the following ordinary differential equations.

\hspace{-.7cm} \begin{minipage}[c]{0.5\textwidth} 
\begin{equation*} \begin{split} 
\dot{x}_1 &= f(x_1,  \textcolor{wred}{x_{2}}, \textcolor{wgreen}{x_3}, \la) \\
{\bf A}: \quad  \dot{x}_2 &= f(x_2, \textcolor{wred}{x_{3}}, \textcolor{wgreen}{x_3}, \la) \quad \\
\dot{x}_3 &=f(x_3,  \textcolor{wred}{x_3}, \textcolor{wgreen}{x_3},  \la)
\end{split}
\end{equation*} 
\end{minipage}
\begin{minipage}[c]{0.5\textwidth} 
\begin{equation*} \ 
\begin{split} 
\dot{x}_1 &= f(x_1,  \textcolor{wred}{x_{2}}, \textcolor{wgreen}{x_2}, \textcolor{blue}{x_3},  \la) \\
{\bf B}: \quad \dot{x}_2 &= f(x_2, \textcolor{wred}{x_{3}}, \textcolor{wgreen}{x_2}, \textcolor{blue}{x_3}, \la) \quad \\
\dot{x}_3 &=f(x_3,  \textcolor{wred}{x_3}, \textcolor{wgreen}{x_2}, \textcolor{blue}{x_3}, \la)  
\end{split}
\end{equation*} 
\end{minipage} \\
\begin{minipage}[c]{0.61 \textwidth} \space
\begin{equation}\label{ex1}
\begin{split}
\dot{x}_1 &= f(x_1,  \textcolor{wred}{x_{2}}, \textcolor{wgreen}{x_1}, \textcolor{blue}{x_3}, \textcolor{brown}{x_2}, \la) \\
{\bf C}: \quad  \dot{x}_2 &= f(x_2, \textcolor{wred}{x_{3}}, \textcolor{wgreen}{x_1}, \textcolor{blue}{x_3}, \textcolor{brown}{x_2}, \la) \\
\dot{x}_3 &=f(x_3,  \textcolor{wred}{x_3}, \textcolor{wgreen}{x_1}, \textcolor{blue}{x_3}, \textcolor{brown}{x_2}, \la)
\end{split}
\end{equation}  
\end{minipage}
\\ \mbox{} \vspace{2mm}\\
Here, $x_1, x_2, x_3\in \R$ describe the states of the cells in the network, while $\la \in \R$ is a parameter. We shall assume that the ``response function'' $f:\R^3 \times \R \rightarrow \R$ is smooth. Note that the network structure does not change as $\la$ varies. Instead, one could think of the response function $f$ as being variable in $\la$. 

The ODEs in (\ref{ex1}) have several properties that distinguish them from arbitrary three-dimensional dynamical systems. First of all, one can observe that setting $x_1=x_2=x_3$ in (\ref{ex1}) yields that $\dot x_1 = \dot x_2=\dot x_3$, and similarly that $x_2=x_3$ implies $\dot x_2=\dot x_3$. This means that 
in all three networks the {\it polydiagonal subspaces} or \textit{synchrony subspaces} 
$$\{x_1 = x_2 = x_3\} \ \mbox{and} \ \{x_2 = x_3\}$$ 
are preserved under the dynamics (i.e. they are flow-invariant). In particular, this is true for any choice of response function $f$, so that these invariant subspaces only depend  on the network structure of the ODEs. One therefore calls them \textit{robust} synchrony spaces. It can also be checked that the above two synchrony spaces are the only such robust synchrony spaces (in all three examples). 

One may now ask how synchronous solutions emerge or disappear in a local bifurcation. We will answer this question in Section \ref{sec6} by means of center manifold reduction, but we shall indicate a few important aspects of this method already here.  First of all, let us assume that 
$$f(0,0)=0\, ,$$ 
so that $x=0$ is a fully synchronous steady state for the parameter value $\la=0$. Center manifold reduction starts with computing the center subspace at $(x,\lambda)=(0,0)$. This space is determined by the Jacobian matrices of the ODEs in (\ref{ex1}). Let us write $\gamma^{\bf i}_f(x,\la)$ (${\bf i} = {\bf A}, {\bf B}, {\bf C}$) for the vector fields at the right hand side of (\ref{ex1}), and let us set  $a := D_{x_1}f(0, 0), \textcolor{wred}{b} := D_{x_2}f(0, 0)$, $\textcolor{wgreen}{c} := D_{x_3}f(0, 0)$, $\textcolor{blue}{d} := D_{x_4}f(0, 0)$ and $\textcolor{brown}{e} := D_{x_5}f(0, 0)$. In terms of these quantities, the Jacobian matrices  are  given by
 
\begin{minipage}[t]{0.45\textwidth}
\begin{align*}
D_x\gamma^{\bf A}_f(0, 0) =  \begin{pmatrix*}[l]
 a & \textcolor{wred}{b} & \textcolor{wgreen}{c}\\
 0 & a & \textcolor{wred}{b} + \textcolor{wgreen}{c} \\
 0 & 0 & a + \textcolor{wred}{b} + \textcolor{wgreen}{c}
 \end{pmatrix*} 
 \end{align*}
\end{minipage}
\begin{minipage}[t]{0.52\textwidth}
 \begin{align*}
D_x\gamma^{\bf B}_f(0, 0) = \begin{pmatrix*}[l]
 a & \textcolor{wred}{b} + \textcolor{wgreen}{c} & \textcolor{blue}{d} \\
 0 & a + \textcolor{wgreen}{c}& \textcolor{wred}{b} +  \textcolor{blue}{d}  \\
 0 & \textcolor{wgreen}{c} & a+ \textcolor{wred}{b} +  \textcolor{blue}{d} 
 \end{pmatrix*} 
 \end{align*}
\end{minipage} \hfill
\begin{minipage}[t]{0.72\textwidth}
 \begin{align}\label{ex2}
D_x\gamma^{\bf C}_f(0, 0) = \begin{pmatrix*}[l]
 a + \textcolor{wgreen}{c} & \textcolor{wred}{b}  +  \textcolor{brown}{e} &  \textcolor{blue}{d}\\
  \textcolor{wgreen}{c} & a + \textcolor{brown}{e} & \textcolor{wred}{b} + \textcolor{blue}{d} \\
  \textcolor{wgreen}{c} &  \textcolor{brown}{e} & a+ \textcolor{wred}{b} + \textcolor{blue}{d}
 \end{pmatrix*}\quad
 \end{align}
\end{minipage} \\
\vspace{3mm}

\noindent We may now observe the remarkable fact that all three Jacobian matrices in (\ref{ex2}) have a double real eigenvalue equal to $a$. If we furthermore assume that $ \textcolor{wred}{b}\neq 0$ and that $\textcolor{wred}{b} + \textcolor{wgreen}{c} \not= 0$ (for network {\bf A}),  $\textcolor{wred}{b} + \textcolor{wgreen}{c}+ \textcolor{blue}{d} \not= 0$ (for network {\bf B}), $\textcolor{wred}{b} + \textcolor{wgreen}{c}+ \textcolor{blue}{d} + \textcolor{brown}{e} \not= 0$ (for network {\bf C}), then this eigenvalue $a$ has algebraic multiplicity two and geometric multiplicity one (this is again true in all three examples). Such a degeneracy in the spectrum is very exceptional among Jacobian matrices of arbitrary ODEs, but here it is forced by the network structure. In particular, it implies that the center manifold of the ODEs is two-dimensional as soon as $a=0$, which in turn indicates that a quite complicated bifurcation may occur.
Using center manifold reduction we shall verify in Section \ref{sec6} that networks {\bf A}, {\bf B} and {\bf C} can generically support precisely one type of steady state bifurcation when the eigenvalue $a$ crosses zero. It is a so-called  {\it synchrony breaking bifurcation} in which a fully synchronous branch, a partially synchronous branch and a fully non-synchronous branch of steady states emerge. Table \ref{tabelsg} lists the asymptotic growth rates of these branches in $\la$, and their possible stability types. Note that although network ${\bf C}$ has identical synchrony and spectral properties as networks ${\bf A}$ and ${\bf B}$, it admits a totally different generic synchrony breaking steady state bifurcation. In particular, in network ${\bf C}$ the stability of the fully synchronous branch may be transferred either to the partially synchronous branch or to the fully non-synchronous branch.  Another curiosity is that the non-synchronous branch of network ${\bf C}$ is  tangential to the space $\{ x_2 = x_3\}$, i.e. it is partially synchronous to first order in $\la$ (this will be shown in Section \ref{sec6}). 
\begin{table}\label{tabelsg} 
\begin{center}
\begin{tabular}{|l || c | c | c |}
\hline
  &  \multicolumn{3}{ | c |}{Branches in examples ${\bf A}$ and ${\bf B}$}   \\ \hline
Synchrony &  Asymptotics & $\la < 0$ & $\la >0$   \\ \hline
Full & $\sim \la$ & $- -$ & $++$   \\ 
Partial & $\sim \la$ & $+-$ & $-+$   \\
None & $\sim \sqrt{\la}$ & & $+-$, $- -$  \\ 
\hline
\end{tabular}
\\ \mbox{} \vspace{2mm}\\
\begin{tabular}{|l || c | c | c || c | c || c | c |}
\hline
  &    \multicolumn{7}{  | c| }{Branches in example ${\bf C}$} \\ \hline
Synchrony &    Asymptotics & $\la < 0$ & $\la >0$ &  $\la < 0$ & $\la >0$ &  $\la < 0$ & $\la >0$ \\ \hline
Full &  $\sim \la$ & $--$& $++$ & $--$ & $++$ & $--$ & $++$  \\ 
Partial &   $\sim \la$& $++$ & $--$ & $+-$ & $-+$ & $+-$ & $-+$ \\
None &   $\sim \la$& $+-$ &$-+$ & $++$ & $--$  & $+-$ & $-+$ \\ 
\hline
\end{tabular}
\end{center}
\caption{Asymptotics in $\la$ of the three branches of steady states that emerge from a synchrony breaking steady state bifurcation in networks ${\bf A}$, ${\bf B}$ and ${\bf C}$. The table also indicates their stability through the signs of two out of three eigenvalues, where for network {\bf C} there are three possible scenarios.}
\end{table}

  We remark that non-trivial invariant subspaces, spectral degeneracies and anomalous bifurcations are all very common in the setting of equivariant dynamics \cite{perspective, golschaef2}, where they are forced by symmetry. On the one hand, it is obvious that networks {\bf A}, {\bf B} and {\bf C} are not symmetric under any permutation of cells. As a result, none of the ODEs in (\ref{ex2}) is equivariant under a linear group action. 
  On the other hand, it was shown in \cite{RinkSanders2} that the robust synchrony spaces, the degenerate spectrum and the unusual bifurcations of networks {\bf A}, {\bf B} and {\bf C} can all be explained from hidden semigroup symmetry.  
  
  For example, the differential equations of network {\bf A} are equivariant under the noninvertible linear map 
$$S: (x_1,x_2,x_3) \mapsto (x_2,x_3,x_3)$$ 
that transforms solutions of the ODEs into solutions. In fact, every vector field that commutes with $S$ is necessarily an admissible vector field for network {\bf A}. This is because network ${\bf A}$ is a so-called fundamental network, see Section \ref{sec2}. Moreover, it is not hard to check that any ODE that admits the symmetry $S$ must have the invariant subspaces $\{x_1=x_2=x_3\}$ and $\{x_2=x_3\}$, and that any matrix that commutes with $S$ must have a double eigenvalue. We will also prove in Section \ref{sec4} that the symmetry $S$ is  inherited by the center manifold of network ${\bf A}$.  The restrictions that  symmetry imposes on the center manifold dynamics force the remarkable synchrony breaking bifurcation of network {\bf A}. 


Similar things can be said for networks {\bf B} and {\bf C}, even though one can show that $\gamma^{\bf B}_f$ and $\gamma^{\bf C}_f$ commute with no linear maps other than the identity. On the other hand, networks {\bf B} and {\bf C} can be realised as quotient networks of networks with semigroup symmetry. In particular, network ${\bf B}$  
is the restriction to the robust synchrony space $\{X_2 = X_3\}$ of the network differential equations

\begin{minipage}[t]{0.55\textwidth}
\begin{equation} \label{vierfun}
\widetilde{\bf B}:\quad 
\begin{array}{ll}
\dot{X}_1 &= f(X_1,  \textcolor{wred}{X_{2}},  \textcolor{wgreen}{X_3}, \textcolor{blue}{X_4}, \la) \\
\dot{X}_2 &= f(X_2,  \textcolor{wred}{X_{4}}, \textcolor{wgreen}{X_3}, \textcolor{blue}{X_4}, \la) \\
\dot{X}_3 &=f(X_3,   \textcolor{wred}{X_4}, \textcolor{wgreen}{X_3}, \textcolor{blue}{X_4}, \la)  \\
\dot{X}_4 &= f(X_4,  \textcolor{wred}{X_{4}}, \textcolor{wgreen}{X_3}, \textcolor{blue}{X_4}, \la) 
\end{array}
\end{equation} 
\end{minipage}
\begin{minipage}[t]{0.35\textwidth} \vspace{-7mm}
\begin{tikzpicture}[scale=1.3]
\tikzstyle{every node} = [ draw, circle, fill={rgb:red,1;green,2;blue,6;white,20}, anchor=center]
\node (2) at ( -1, 1) {2};
\node (1) at ( 1, 1) {1};
\node (4) at (1,-1) {4};
\node (3) at (-1,-1) {3};
\node (B) [draw=none, fill=none] at (0, -0.3) {\huge{$\widetilde{\bf B}$}};
\path[ ->, line width=1.7pt,  color={rgb:red,12; blue,3; green, 2 ;white,15}]
	(2) edge[bend left=20] (1)
	(4) edge[bend left=-10] (2)
	(4) edge[bend left=20] (3)
	(4) edge[loop, in=-50, out=10,  distance=7mm] (4);
\path[ ->, ,line width=1.7pt, color={rgb:red,1; blue,1; green, 6 ;black,4;white,18}]
	(3) edge[bend right=-10] (1)
	(3) edge[bend left=20 ] (2)
	(3) edge[bend left=0 ] (4)
	(3) edge[loop, in=190, out=250, min distance = 7mm ] (3);
\path[ ->, line width=1.7pt, color={blue}]
	(4) edge[bend left=-20] (1)
	(4) edge[bend left=-40] (2)
	(4) edge[bend left=-20] (3)
	(4) edge[loop, in=-120, out=-90,  distance=7mm] (4);
\end{tikzpicture}
\end{minipage}
\\ \mbox{}\\
\noindent These differential equations commute with the two noninvertible linear maps 
\begin{equation}
\begin{split}
&(X_1, X_2, X_3, X_4) \mapsto (X_2, X_4, X_3, X_4)\, , \\ 
&(X_1, X_2, X_3, X_4) \mapsto (X_3, X_4, X_3, X_4) \, .
\end{split}
\end{equation}
Conversely, every ODE that is  equivariant under  these two symmetries is necessarily of the  form (\ref{vierfun}) for some $f(X, \lambda)$, i.e. it is  admissible for network $\widetilde{\bf B}$. We call network $\widetilde{\bf B}$  the \textit{fundamental network} of network ${\bf B}$. It was shown in \cite{fibr} that every homogeneous network is the quotient of such a fundamental network with a semigroup of symmetries. We will recover this fact in Section \ref{sec2}. 

It turns out that the fundamental network of ${\bf C}$ is given by

\begin{minipage}[t]{0.6\textwidth}
\begin{equation} \label{vijffun}
\widetilde{\bf C}:\quad 
\begin{array}{ll}
\dot{X}_1 &= f(X_1,  \textcolor{wred}{X_{2}}, \textcolor{wgreen}{X_3}, \textcolor{blue}{X_4},  \textcolor{brown}{X_5}, \la) \\
\dot{X}_2 &= f(X_2, \textcolor{wred}{X_{4}}, \textcolor{wgreen}{X_3},  \textcolor{blue}{X_4}, \textcolor{brown}{X_5}, \la) \\
 \dot{X}_3 &=f(X_3,  \textcolor{wred}{X_5}, \textcolor{wgreen}{X_3},  \textcolor{blue}{X_4}, \textcolor{brown}{X_5}, \la)  \\
\dot{X}_4 &=f(X_4, \textcolor{wred}{X_{4}}, \textcolor{wgreen}{X_3},  \textcolor{blue}{X_4}, \textcolor{brown}{X_5}, \la) \\
\dot{X}_5 &=f(X_5, \textcolor{wred}{X_{4}}, \textcolor{wgreen}{X_3},  \textcolor{blue}{X_4}, \textcolor{brown}{X_5}, \la) 
\end{array}
\end{equation} 
\end{minipage}
\begin{minipage}[t]{0.35\textwidth} \vspace{-7mm}
\begin{tikzpicture}[scale=1.3]
\tikzstyle{every node} = [ draw, circle, fill={rgb:red,1;green,2;blue,6;white,20}, anchor=center]
\node (1) at ( 0, 1.414) {1};
\node (5) at ( 1.345, 0.437) {5};
\node (3) at (-0.831,-1.144) {3};
\node (4) at (0.831,-1.144) {4};
\node (2) at ( -1.345, 0.437) {2};
\node (C) [draw=none, fill=none] at (0, 0.5) {\huge{$\widetilde{\bf C}$}};
\path[ ->, line width=1.7pt, color={rgb:red,12; blue,3; green, 2 ;white,15}]
	(2) edge[bend left=20] (1)
	(4) edge[bend left=-20] (5)
	(4) edge[bend left=-10] (2)
	(5) edge[bend left=10] (3)
	(4) edge[loop, in=310, out=0,  distance=7mm] (4);
\path[ ->, ,line width=1.7pt, color={rgb:red,1; blue,1; green, 6 ;black,4;white,18}]
	(3) edge[bend right=-30] (1)
	(3) edge[bend right=-10] (5)
	(3) edge[bend left=20 ] (2)
	(3) edge[bend left=-25 ] (4)
	(3) edge[loop, in=190, out=250, min distance = 7mm ] (3);
\path[ ->, ,line width=1.7pt, color={blue}]
	(4) edge[bend right=20] (1)
	(4) edge[bend right=-5] (5)
	(4) edge[bend left=20 ] (2)
	(4) edge[bend left=5 ] (3)
	(4) edge[loop, in=240, out=290, min distance = 7mm ] (4);
\path[ ->, ,line width=1.7pt, color={brown}]
	(5) edge[bend right=20] (1)
	(5) edge[bend right=-50] (4)
	(5) edge[bend left=-30 ] (2)
	(5) edge[bend left=-30 ] (3)
	(5) edge[loop, in=-10, out=40, min distance = 7mm ] (5);		
\end{tikzpicture}
\end{minipage}
\\ \mbox{}\\
\noindent
Indeed, network ${\bf C}$ arises as the restriction of network $\widetilde{\bf C}$ to the robust synchrony space $\{X_1 = X_3, X_2 = X_5\}$. Moreover, the equations of motion (\ref{vijffun}) of network $\widetilde{\bf C}$ are precisely the equivariant ODEs for the noninvertible linear maps  
\begin{equation}
\begin{split}
&(X_1, X_2, X_3, X_4, X_5) \mapsto (X_2, X_4, X_3, X_4, X_5)\, , \\
&(X_1, X_2, X_3, X_4, X_5) \mapsto (X_3, X_5, X_3, X_4, X_5) \, .
\end{split}
\end{equation}
The symmetries of networks $\widetilde{\bf B}$ and $\widetilde{\bf C}$ are inherited by their center manifolds. We will see in Sections \ref{sec5} and \ref{sec6}  how  they in turn affect the center manifolds of  {\bf B} and {\bf C},  thus forcing the anomalous bifurcations in these two original networks.

\section{Homogeneous and fundamental networks}\label{sec2}
In this section we give  a short overview of the results and definitions in \cite{fibr, RinkSanders3, CCN}. 
We shall be concerned with ODEs of the general form
\begin{align}\label{first}
\begin{split}
\dot{x}_1 &= f(x_{\sigma_1(1)}, \dots x_{\sigma_n(1)}) \\
\dot{x}_2 &= f(x_{\sigma_1(2)}, \dots x_{\sigma_n(2)}) \\
 &\vdots \\
\dot{x}_N &=f(x_{\sigma_1(N)}, \dots x_{\sigma_n(N)}) 
\end{split}\, .
\end{align}
Here, every variable $x_j$ takes values in the same vector space $V$ and can be though of as the state of cell $\#j$ in a network. For every $i \in \{1, \dots n\}$, 
$$\sigma_i: \{1, \dots N\}\to \{1, \dots N\}$$ is a function from the collection of cells of the network to itself. Intuitively, these functions may be thought of as representing the different types of input in the network. In particular, if $i \in \{1, \dots n\}$ and $j,k \in \{1, \dots N\}$ are such that $\sigma_i(j) = k$, then this is to be interpreted as cell $\#j$ receiving an input of type $i$ from cell $\#k$. Note that there is no reason to assume that any of the functions $\sigma_i$ is a bijection. 

 The way the inputs of a cell are processed is  determined by the properties of the response function $f: V^n \rightarrow V$, whose different arguments distinguish  different types of input. Note that the same response function appears in every component of  (\ref{first}), meaning that every cell responds equally to its inputs. This may be interpreted as the cells being identical. We therefore say that (\ref{first}) represents a \textit{homogenous coupled cell network}.  
Another assumption we will make is that the total set of input functions 
$$\Sigma := \{\sigma_1, \ldots, \sigma_n\}$$ 
is closed under composition of maps. This is no restriction because one may  add compositions of input functions to $\Sigma$ until this process terminates, see \cite{CCN}. Furthermore, enlarging $\Sigma$ only enlarges the class of admissible vector fields.  Being closed under composition, $\Sigma$ has the structure of a semigroup. To model internal dynamics, we will moreover assume without loss of generality that $\sigma_1$ is the identity on $ \{1, \dots N\}$, making $\Sigma$ in fact a monoid. For $f:V^n\to V$, we will then denote the vector field at the right hand side of equation (\ref{first}) by $$\gamma_f: V^N \rightarrow V^N\, .$$

\begin{ex}\label{inputmapsexample}
Networks {\bf A}, {\bf B} and {\bf C} are examples of homogeneous networks. The maps $\sigma_1, \sigma_2, \sigma_3, \sigma_4, \sigma_5$ are given in this case by
\begin{equation}
 \begin{aligned}\nonumber
 \begin{array}{c|ccc} {\rm \bf A} & 1 & 2 & 3 \\ \hline 
\sigma_1 & 1 & 2  & 3   \\
\sigma_2 & 2 & 3 & 3  \\
\sigma_3 & 3 & 3 & 3  
\end{array} \hspace{.5cm}
   \begin{array}{c|ccc} {\rm \bf B} & 1 & 2 & 3 \\ \hline 
\sigma_1 & 1 & 2  & 3   \\
\sigma_2 & 2 & 3 & 3  \\
\sigma_3 & 2 & 2 & 2 \\  
\sigma_4 & 3 & 3  & 3   
\end{array}
 \hspace{.5cm}
 \begin{array}{c|ccc} {\rm \bf C} & 1 & 2 & 3 \\ \hline 
\sigma_1 & 1 & 2  & 3   \\
\sigma_2 & 2 & 3 & 3 \\
\sigma_3 & 1 & 1 & 1  \\
\sigma_4 & 3 & 3  & 3   \\
\sigma_5 & 2 & 2 & 2  
\end{array}
 \ .
\end{aligned}
\end{equation}
For all three networks, these maps are closed under composition, i.e. they form a semigroup $\Sigma$. \hfill $\triangle$
\end{ex}

\begin{defi}
Let $P = \{P_i\}_{i=1}^r$, $P_i \subset \{1, 2, \ldots, N\}$ be a partition of the collection of nodes of a homogeneous network. The \textit{synchrony space} or \textit{polydiagonal space} corresponding to the partition $P$ is the subspace
$$\syn_P:=\{ x_i = x_j\ \mbox{if} \ i\ \mbox{and}\ j\ \mbox{are in the same element of the partition}\ P\} \subset V^N\, .$$ 
 A synchrony space is called \textit{robust} if for every $f:V^n \rightarrow V$ we have $\gamma_f(\syn_P)\subset \syn_P$, i.e. that it is an invariant space for every $\gamma_f$. 
\end{defi}

\noindent It was shown in \cite{CCN} that adding compositions $\sigma_i \circ \sigma_j$ to $\Sigma$ so as to make $\Sigma$ closed under composition does not affect the set of robust synchrony spaces of the network.

The idea is now to define a bigger network that contains the original network (\ref{first}) as a robust synchrony space. It turns out that the admissible vector fields of this so-called fundamental network are precisely the  equivariant vector fields for the regular representation of the monoid $\Sigma$.

\begin{defi}
Assume that $\Sigma$ has been completed to a monoid, let $n = \#\Sigma$, and let $f:V^n\to V$. The \textit{fundamental network vector field} $\Gamma_f$ of the network  vector field $\gamma_f$ is the vector field on $\bigoplus_{\sigma_i \in \Sigma} V =V^n$ defined by

\begin{equation}
(\Gamma_f)_{\sigma_i} = f \circ A_{\sigma_i} \, .
\end{equation}
Here the linear maps $A_{\sigma_i}: V^n \rightarrow V^n$ are defined by 

\begin{equation}
(A_{\sigma_i}X)_{\sigma_j} := X_{\sigma_j \circ \sigma_i} \, .
\end{equation}
\end{defi}

\noindent It was shown in \cite{fibr} that $\Gamma_f$ is an admissible vector field for the homogeneous network that has the  elements of $\Sigma$ as its cells, and an arrow of type $i$ from cell $\sigma_k$ to cell $\sigma_j$ if $\sigma_i\circ \sigma_j=\sigma_k$. This latter network can be thought of as a Cayley graph of $\Sigma$, see \cite{fibr}.

\begin{thr}\label{repre}
The linear maps $\{A_{\sigma_i}\}_{\sigma_i \in \Sigma}$ form a representation of the monoid $\Sigma$. That is, we have $A_{\sigma_i} \circ A_{\sigma_j} = A_{\sigma_i \circ \sigma_j}$ for all $\sigma_i, \sigma_j \in \Sigma$ and $A_{\sigma_{1}} = {\rm Id}$.
\end{thr}

\begin{proof}
Because $\sigma_1 \circ \sigma_i = \sigma_i \circ \sigma_1 = \sigma_i$ for all $\sigma_i \in \Sigma$, it is clear that $A_{\sigma_{1}} = {\rm Id}$. Furthermore, given $\sigma_i, \sigma_j \in \Sigma$ we have

\begin{equation}
\begin{aligned}
(A_{\sigma_i} A_{\sigma_j} X)_{\sigma_k} &=  (A_{\sigma_i}[ A_{\sigma_j} X])_{\sigma_k} = \\
(A_{\sigma_j} X)_{\sigma_k \circ \sigma_i} &= X_{(\sigma_k \circ \sigma_i) \circ \sigma_j} =\\
X_{\sigma_k \circ (\sigma_i \circ \sigma_j)} &= (A_{\sigma_i \circ \sigma_j}X)_{\sigma_k}  \, ,
\end{aligned}
\end{equation}
for all $X \in V^n$. This proves the statement.
\end{proof}

\begin{thr}\label{symfun}
A vector field $F: V^n \rightarrow V^n$ is of the form $F = \Gamma_f$ for some $f: V^n \rightarrow V$ if and only if we have $F \circ A_{\sigma_i} = A_{\sigma_i} \circ F$ for all $\sigma_i \in \Sigma$.
\end{thr}

\begin{proof}
We will first show that $\Gamma_f \circ A_{\sigma_i} = A_{\sigma_i} \circ \Gamma_f$ for all $f: V^n \rightarrow V$ and $\sigma_i \in \Sigma$. We see that on the one hand side we have

\begin{equation}
\begin{aligned}
[(\Gamma_f  \circ A_{\sigma_i})(X)]_{\sigma_k} = [\Gamma_f ( A_{\sigma_i} X)]_{\sigma_k} = (f \circ A_{\sigma_k} \circ A_{\sigma_i})(X) \, .
\end{aligned}
\end{equation}
On the other, we see that

\begin{equation}
\begin{aligned}
[(A_{\sigma_i} \circ \Gamma_f )(X)]_{\sigma_k} &= [\Gamma_f(X)]_{\sigma_k \circ \sigma_i}  =\\
(f \circ A_{\sigma_k \circ \sigma_i})(X) &= (f \circ A_{\sigma_k} \circ A_{ \sigma_i})(X) \, ,
\end{aligned}
\end{equation}
where in the last step we have used the result of theorem \ref{repre}. This proves the first part of the theorem. 

As for the second, suppose that 
$F \circ A_{\sigma_i} = A_{\sigma_i} \circ F$ for all $\sigma_i \in \Sigma$. Using the definition of $A_{\sigma_i}$ and the fact that $\sigma_1 \circ \sigma_i = \sigma_i$ for all $\sigma_i \in \Sigma$, we see that

\begin{equation}
\begin{aligned}
&[F(X)]_{\sigma_i} = [(A_{\sigma_i} \circ F)(X)]_{\sigma_1} =\\
&[(F \circ A_{\sigma_i})(X)]_{\sigma_1} = (F_{\sigma_1} \circ A_{\sigma_i})(X) \, ,
\end{aligned}
\end{equation}
for all $X \in V^n$. Hence we see that $F = \Gamma_f$ for $f = F_{\sigma_1}$. This proves the second part of the theorem.
\end{proof}
\noindent The following theorem provides the relation between the original network $\gamma_f$ and the new network $\Gamma_f$.
\begin{thr}\label{inbed}
For any node $p \in \{1, \dots N\}$, define the map $\pi_p: V^N \rightarrow V^n$  by 
$$\pi_p (x)_{\sigma_j} := x_{\sigma_j(p)}\, .$$ 
Then $\pi_p$ is a semiconjugacy between  $\gamma_f$ and $\Gamma_f$. That is, we have $$\pi_p \circ \gamma_f = \Gamma_f \circ \pi_p\, .$$
\end{thr}
\noindent To prove theorem \ref{inbed} we  need the following lemma.

\begin{lem}\label{lemb}
For any node $p \in \{1, \dots N\}$ and  input function $\sigma_i$, we have

\[A_{\sigma_i} \circ \pi_p = \pi_{\sigma_i(p)} \, .\]
\end{lem}

\begin{proof}
For any $x \in V^N$ we have

\begin{equation}
\begin{aligned}
[(A_{\sigma_i} \circ \pi_p)(x)]_{\sigma_j} &= [A_{\sigma_i} [ \pi_p(x)]]_{\sigma_j} = \\
[\pi_p(x)]_{\sigma_j \circ \sigma_i} &= x_{(\sigma_j \circ \sigma_i)(p)} =\\
x_{\sigma_j (\sigma_i(p))} &= [\pi_{\sigma_i(p)}(x)]_{\sigma_j} \, .
\end{aligned}
\end{equation}
This proves the statement.
\end{proof}

\begin{proof}[Proof of theorem \ref{inbed}]
Recall that 
$(\gamma_f(x))_p = f(x_{\sigma_1(p)}, \dots x_{\sigma_n(p)})$ by definition, from which it follows that we may in fact write $(\gamma_f)_p = f \circ \pi_p$. 

On the one hand side we therefore have

\begin{equation}
\begin{aligned}
[(\pi_p \circ \gamma_f)(x)]_{\sigma_i} = (\gamma_f)_{\sigma_i(p)}(x) = (f \circ \pi_{\sigma_i(p)})(x) \, .
\end{aligned}
\end{equation}
On the other hand, we see by lemma \ref{lemb} that

\begin{equation}
\begin{aligned}
[(\Gamma_f \circ \pi_p)(x)]_{\sigma_i} = (f \circ A_{\sigma_i})(\pi_p(x)) =\\
(f \circ A_{\sigma_i} \circ \pi_p)(x) = (f \circ  \pi_{\sigma_i(p)})(x) \, .
\end{aligned}
\end{equation}
This proves the theorem.
\end{proof}

\begin{remk}
Note that the map $\pi_p$ is injective if and only if 
$$\{\sigma_i(p): \sigma_i \in \Sigma\} = \{1, \dots N\}\, .$$ 
This is to be interpreted as the cell $p$ being influenced by every other cell in the network. In particular, it is  natural to assume that at least one such cell exists in the (original) network. In that case, the dynamics of $\gamma_f$ is embedded in that of $\Gamma_f$ as the restriction of $\Gamma_f$ to the space 
$$\{X_{\sigma_i} = X_{\sigma_j} \text{ if } \sigma_i(p) = \sigma_j(p)\}$$ 
for any such node $p$ for which $\pi_p$ is injective. Note that this space is a polydiagonal space. Furthermore, since it is invariant  for every $f$, we  conclude that this space is in fact a robust synchrony space of the fundamental network.
 \hfill$\triangle$
\end{remk}

\section{Center manifold reduction for  networks}\label{sec3}

In this section we shall describe the main result of this paper. We start with a well-known theorem on the existence of a local  invariant manifold near every steady state of an ODE. The most important feature of this so-called {\it center manifold} is that it contains all bounded (small) solutions, such as steady state points and (small) periodic orbits. We then generalise this result to the setting of fundamental networks, in a way that allows us to retain their symmetries. Because we know from theorem \ref{symfun} that these symmetries completely describe the fundamental network vector field, this will in turn allow us to give a full description of the vector fields that one obtains after restricting to the center manifold. 

Let us first consider differential equations of the general form

\begin{equation}\label{eq:4.1}
\dot{x} = F(x) \, ,
\end{equation}
where \(F:\R^n  \rightarrow \R^n\) is of class \( C^k\) for some \(k \geq 1 \) and satisfies \( F(0) = 0\). Without loss of generality, we may write

\begin{equation}\label{eq:4.2}
\dot{x} = Ax + G(x) \, .
\end{equation}
Here, \(A = DF(0) \), from which it follows that \(G:\R^n  \rightarrow \R^n\) is again of class \( C^k\) and satisfies  \(G(0) = 0\) and \(DG(0) = 0 \). Let us furthermore denote by \(X_c \) the center subspace of $A$. That is, \(X_c \) is the span of the generalized eigenvectors corresponding to the purely imaginary eigenvalues of \(A\). Likewise, we denote by \(X_h \) the hyperbolic subspace of \(A\), which corresponds to the remaining eigenvalues. These two spaces complement each other in $\R^n$, i.e. we have

\begin{equation}\label{eq:4.3}
 \R^n = X_c \oplus X_h \, .
\end{equation}
 Finally, let \(\pi_c\) and  \(\pi_h\) be the projections onto \(X_c \) respectively \(X_h\) corresponding to this decomposition. The following theorem is well-known.

\begin{thr}[{\bf Center Manifold Reduction}] \label{T1}
Given \(A \in \mathcal{L}(\R^n)\) and $k \in \N$, there exists an \(\epsilon = \epsilon(A,k) > 0 \) such that the following holds: If \(G:\R^n  \rightarrow \R^n\) is of class \( C^k\) with $G(0) = 0$ and $DG(0) = 0$ and furthermore satisfies

\begin{itemize}
\item $\displaystyle \sup_{x \in \R^n}||D^jG(x)|| < \infty$ for $0 \leq j \leq k$,
\item $\displaystyle \sup_{x \in \R^n}||DG(x)|| < \epsilon$,
\end{itemize}
then there exists a function $\psi: X_c \rightarrow X_h$ of class \( C^k\) such that its graph in $\R^n$ is an invariant manifold for the system \eqref{eq:4.1}. More precisely, we have
\begin{equation}\label{eq:4.4}
M_c:= 
\{ x_c + \psi(x_c) : x_c \in X_c \} = \{x \in \R^n : \sup_{t \in \R} ||\pi_h\phi^t(x)|| <  \infty \} \, .
\end{equation}
Here $\phi^t$ denotes the flow of equation \eqref{eq:4.1}.  The function $\psi$ satisfies $\psi(0) = 0$ and $D\psi(0) = 0$.

$M_c$ is called the (global) center manifold of \eqref{eq:4.1}. In particular, it contains all bounded solutions to  \eqref{eq:4.1}, such as steady state points and periodic solutions. \end{thr}
\noindent A comprehensive proof of this theorem can be found in \cite{vdbauw}. This reference also describes a way around the seemingly strict conditions on the size of the nonlinearity $G$ and its derivatives: if $G$ does not satisfy these conditions, then one simply  multiplies it by a real-valued bump function with small enough support. Since in bifurcation theory one is generally only interested in orbits close to the bifurcation point, this is often a viable solution. 

Moreover, if the vector fields $F$ and $G$ are equivariant under the action of some compact group $\mathcal{G}$, then this bump function can be chosen invariant under this action. As a result, the center manifold is $\mathcal{G}$-invariant. To make this more precise, let us assume that the action of $\mathcal{G}$ is by linear maps $\{A_{g}: g \in \mathcal{G}\}$. This will for example be the case after applying Bochner's linearisation theorem. Let us furthermore denote by $B_r$ an open ball of radius $r > 0$ in $\R^n$ centered around the origin and by $\chi$ a smooth ``bump function'' from $ \R^n$ to $\R$ that takes the value $1$ inside $B_1$ and $0$ outside $B_{2}$. It can then be shown that the vector field

\begin{equation}\label{eq:4.5}
\widetilde{G}_{\rho}(x) := \chi(\rho^{-1}x)G(x) 
\end{equation}
satisfies the necessary bounds of theorem \ref{T1} for small enough $\rho > 0$. However, this function will in general not be $\mathcal{G}$-equivariant anymore, as the bump function $\chi$ may not be $\mathcal{G}$-invariant. Instead, we may define a new bump function

\begin{equation}\label{eq:4.6}
\overline{\chi}(x) := \int_{\mathcal{G}}  \chi(A_gx) d\mu \, ,
\end{equation}
where $d\mu$ denotes the normalised Haar measure on $\mathcal{G}$ (or simply the normalised counting measure, if $\mathcal{G}$ is finite). The function $\overline{\chi}$ is now $\mathcal{G}$-invariant by construction. From this  it follows that 
  
\begin{equation}\label{eq:4.7}
{G}_{\rho}(x) := \overline{\chi}(\rho^{-1}x)G(x) 
\end{equation}
is $\mathcal{G}$-equivariant, because  

\begin{equation}\label{eq:4.8}
\begin{aligned}
{G}_{\rho}(A_gx)= & \ \overline{\chi}(\rho^{-1}A_gx)G(A_gx) = \overline{\chi}(A_g\rho^{-1}x)A_gG(x) =\\
 &\overline{\chi}(\rho^{-1}x)A_gG(x) = A_g\overline{\chi}(\rho^{-1}x)G(x) =  A_g{G}_{\rho}(x) \, .
\end{aligned}
\end{equation}
By the compactness of $\mathcal{G}$ we may also assume without loss of generality that $\mathcal{G}$ acts by isometries. That is, we have $||A_gx|| = ||x||$ for all $g \in \mathcal{G}$ and $x \in \R^n$.  In particular, we see that $\overline{\chi}$ again takes the value $1$ inside $B_1$ and vanishes outside $B_{2}$. Hence, as is the case for $\widetilde{G}_{\rho}$, we may conclude that ${G}_{\rho}$ satisfies the necessary bounds of theorem \ref{T1} for small enough $\rho > 0$. It then follows from the equivariance of $F$ and ${G}_{\rho}$ that center manifold reduction can in fact be done in an equivariant manner, meaning that the function $\psi: X_c \rightarrow X_h$ is equivariant and that $M_c$ is invariant under the symmetries.

Unfortunately we cannot apply the same procedure in the setting of  networks and fundamental networks, as it relies heavily on the symmetries $A_g$ being invertible (for example in the existence of an invariant measure). As an example, we note that any function $\chi: \R^3 \rightarrow \R$ that is invariant under the symmetry $(X_1,X_2,X_3 ) \mapsto (X_2, X_3,X_3)$ of example ${\bf A}$ would necessarily be constant along the line $\{X_2 = X_3 = 0\}$. It is clear that this would exclude any non-trivial bump function centered around the origin. 
Instead, we will show that one can replace the function $f$ in $\Gamma_f$ in a way to make $\Gamma_f$ satisfy the necessary bounds. Note that in this way the symmetries of $\Gamma_f$ are not broken.

To formalise this procedure, let us first describe our setting a bit more accurately. We want to study bifurcation problems, so we will assume from now on that the response function $f$ depends on parameters, i.e. we assume that
$$f: V^n \times \Omega \to V \ \mbox{with}\ \Omega \subset \R^l$$
is a smooth function of the network states and of parameters \(\la \in \Omega\). For the purpose of center manifold reduction, it is  useful to view these parameters as variables of the ODEs, i.e. to consider the augmented network equations

\begin{equation}\label{eq:4.12}
 \left( \begin{array}{c} 
\dot{x} \\
\dot{\lambda} 
\end{array}
\right) =
\left( \begin{array}{c} 
\Gamma_f(x, \la)\\
0 
\end{array}
\right)
\, ,
\end{equation}
with $\Gamma_f$ defined as before by $$\Gamma_f(x, \lambda)_{\sigma_i}:= f(A_{\sigma_i}x, \lambda)\, .$$
We will set $\underline{x} := (x, \la) \in V^n \times \Omega $ and $\underline{\Gamma}_f := (\Gamma_f,0) : V^n \times \Omega \rightarrow V^n \times \Omega$, and will henceforth abbreviate  equation (\ref{eq:4.12}) as

\begin{equation}\label{eq:4.13}
\dot{\underline{x}} = \underline{\Gamma}_f(\underline{x}) \, .
\end{equation}
It is clear that this system is now equivariant under symmetries of the form

\[\ul{A}_{\sigma_i} : (x,\la) \mapsto (A_{\sigma_i}x, \la) \text{ for } i \in \{1 \dots n\} \, . \]
Furthermore, note that in this notation we also have

\begin{equation}\label{eq:4.14} 
\begin{aligned}
(\underline{\Gamma}_f )_i &= f \circ \ul{A}_{\sigma_i} &&\text{ for } i \in \{1 \dots n\} \, , \\
(\underline{\Gamma}_f )_i &= 0 &&\text{ for } i = n+1\, ,
\end{aligned}
\end{equation}
where we denote by $\ul{x}_{n+1}$ the $\la$-part of the vector  $\ul{x} = (x, \la) \in V^n \times \Omega$. Following the setting of theorem \ref{T1}, we can write $\underline{\Gamma}_f(\underline{x})$ as

\begin{equation}\label{eq:4.15}
\underline{\Gamma}_f(\underline{x}) = D\underline{\Gamma}_f(0)\ul{x} + G(\ul{x}) \, ,
\end{equation}
where $G: V^n \times \Omega \rightarrow V^n \times \Omega$ satisfies $G(0) = 0$ and  $DG(0) = 0$. The first thing to note is that $D\underline{\Gamma}_f(0)\ul{x}$ is again of the form $\ul{\Gamma}_h(\ul{x})$, namely for 
\begin{equation}\label{eq:4.16}
h(\ul{x}) = Df(0)\ul{x} = \sum_{k=1}^{n+1} D_{k}f(0){\ul{x}}_k \, .
\end{equation} 
Indeed, for $i \in \{1, \dots n\}$ we have

\begin{equation}\label{eq:4.17}
\begin{aligned}
(D\ul{\Gamma}_f(0)\ul{x})_i &=&&\sum_{j=1}^{n+1}{D\ul{\Gamma}_f(0)}_{i,j}\ul{x}_{j} = \\
& &&\sum_{j=1}^{n+1} D_j(f \circ \ul{A}_{\sigma_i})(0)\ul{x}_j =\\
& &&\sum_{j=1}^{n+1} \sum_{k=1}^{n+1} D_kf(0)  (\ul{A}_{\sigma_i})_{k,j}\ul{x}_j= \\
& &&\sum_{k=1}^{n+1} D_kf(0) (\ul{A}_{\sigma_i} \ul{x})_{k} = (\ul{\Gamma}_h(\ul{x}))_i \, , 
\end{aligned}
\end{equation}
whereas 

\begin{equation}\label{eq:4.18}
(D\ul{\Gamma}_f(0)\ul{x})_{n+1} = \sum_{j=1}^{n+1}{D\ul{\Gamma}_f(0)}_{n+1,j}\ul{x}_{j} = 0 \, .
\end{equation}
It follows that we may write $G(\ul{x}) = \ul{\Gamma}_f(\ul{x}) - \ul{\Gamma}_h(\ul{x}) = \ul{\Gamma}_g(\ul{x})$, where $g(\ul{x})$ equals $f(\ul{x}) - h(\ul{x}) = f(\ul{x}) - Df(0)\ul{x}$. In particular, assuming that  $f(0) := f(0,0) = 0$, we see that $g(0)=0$ and $Dg(0) = 0$. Summarising, we have the following equivalent of \eqref{eq:4.2}:
\begin{equation}\label{eq:4.19}
\underline{\Gamma}_f(\underline{x}) = D\underline{\Gamma}_f(0)\ul{x} + \ul{\Gamma}_g(\ul{x}) \text{ with } g(0) = 0 \text{ and } Dg(0) = 0 \, .
\end{equation}
We can now proceed to adapt $\ul{\Gamma}_g(\ul{x})$ so as to make it satisfy the conditions of theorem \ref{T1}. To this end, we define $B_r$ to be an open ball in $V^n \times \Omega$ with radius $r$ centered around the origin. Furthermore, let $\chi(\ul{x})$  be a smooth function from $V^n \times \Omega$ to $\R$ that takes the value $1$ inside $B_1$ and $0$ outside $B_2$. Analogous to the procedure for general vector fields, we now set $g_{\rho}(\ul{x}) := \chi(\rho^{-1}\ul{x})g(\ul{x})$ for $\rho \in \R_{>0}$, which equals $g$ inside $B_{\rho}$ and which vanishes outside  $B_{2\rho}$. The following two theorems assure us that the system given by

\begin{equation}\label{eq:4.20}
\dot{\underline{x}} = D\underline{\Gamma}_f(0)\ul{x} + \ul{\Gamma}_{g_{\rho}}(\ul{x}) 
\end{equation}
satisfies the necessary conditions of theorem \ref{T1} for small enough $\rho$, yet agrees with our initial system \ref{eq:4.13} in a small enough neighborhood around the origin.

\begin{thr}\label{T2}
For any function $g:V^n \times \Omega\to V$ and any $\rho > 0$, there exists an open neighborhood in $V^n \times \Omega$ centered around the origin on which $\ul{\Gamma}_{g}(\ul{x})$ and $\ul{\Gamma}_{g_{\rho}}(\ul{x})$ agree.
\end{thr}

\begin{proof}
Remember that we have $\underline{\Gamma}_{g}(\ul{x})_{n+1} = \underline{\Gamma}_{g_{\rho}}(\ul{x})_{n+1} = 0$ for every $\ul{x} \in V^n \times \Omega$, hence there is nothing to check here. For $i \not= n+1$, the $i$-th component of  $\underline{\Gamma}_{g}(\ul{x})$ equals $g\circ \ul{A}_{\sigma_i}(\ul{x})$, whereas that of  $\underline{\Gamma}_{g_{\rho}}(\ul{x})$ equals $g_{\rho}\circ \ul{A}_{\sigma_i}(\ul{x})$. Because $g$ and $g_{\rho}$ agree on $B_{\rho}$, these components are equal on the set $\ul{A}_{\sigma_i}^{-1} (B_{\rho})$, which by the linearity of $\ul{A}_{\sigma_i}$ is an open set containing $0$. The required neighborhood is then obtained by taking the intersection of these sets for the different values of $i$. 
\end{proof}

\begin{thr}\label{T3}
Let $g: V^n \times \Omega \rightarrow V$ be of class $C^k$ for some $k > 0$. For all $\rho > 0$ and $0 \leq j \leq k$ we have 

\begin{equation}\label{eq:4.21}
\sup_{\ul{x} \in V^n \times \Omega}||D^j\ul{\Gamma}_{g_{\rho}}(\ul{x})|| < \infty \, .
\end{equation}
If $g$ furthermore satisfies $g(0) = 0$ and $Dg(0) = 0$, then 
\begin{equation}\label{eq:4.22}
 \lim_{\rho \downarrow 0} \sup_{\ul{x} \in V^n \times \Omega}\hspace{0.3 cm} || D\ul{\Gamma}_{g_{\rho}} (\ul{x})||  = 0 \, .
\end{equation}
\end{thr}

\begin{proof}
We start with the claim on boundedness. It is clear that we only need to show this for the separate components of $\ul{\Gamma}_{g_{\rho}}(\ul{x})$ and their derivatives. However, writing $g_{\rho} = H$ we see that every (non-trivial) component of $\ul{\Gamma}_{g_{\rho}}(\ul{x})$ can be written in  the general form 

\begin{equation}\label{eq:4.23}
\ul{\Gamma}_{g_{\rho}}(\ul{x})_i = (H \circ \ul{A}_{\sigma_i}) (\ul{x})\, ,
\end{equation}
where $H$ is a $C^k$-function with compact support in $V^n \times \Omega$. It is clear that any function that can be written in this way is uniformly bounded. Moreover, taking the derivative gives
\begin{equation}\label{eq:4.24}
D(H \circ \ul{A}_{\sigma_i})(\ul{x})  = DH(\ul{A}_{\sigma_i}\ul{x})\cdot \ul{A}_{\sigma_i} = ((DH\cdot \ul{A}_{\sigma_i})\circ \ul{A}_{\sigma_i })(\ul{x})  \, ,
\end{equation}
which is again of the form (\ref{eq:4.23}), where our new $H$ is now given by the 
$C^{k-1}$-function $DH\cdot \ul{A}_{\sigma_i}$. We conclude by induction that indeed the first $k$ derivatives of $\ul{\Gamma}_{g_{\rho}}$ are uniformly bounded. This proves the first part of the theorem. 

As for the second claim, let $i \in \{1,\dots n\}$, $j\in \{1, \dots n+1\}$ and $\rho > 0$. Then
\begin{equation}
\begin{split}
&\sup_{\ul{x} \in  V^n \times \Omega} || D\ul{\Gamma}_{g_{\rho}} (\ul{x})_{i,j}||  = \\
&\sup_{\ul{x} \in  V^n \times \Omega} || D_j(\ul{\Gamma}_{g_{\rho}})_i(\ul{x})||  = \\
&\sup_{\ul{x} \in  V^n \times \Omega} || D_j (g_{\rho} \circ \ul{A}_{\sigma_i})(\ul{x}) || = \\
&\sup_{\ul{x} \in  V^n \times \Omega} || \sum_{k=1}^{n+1}D_k g_{\rho} (\ul{A}_{\sigma_i}(\ul{x})) (\ul{A}_{\sigma_i})_{k,j}|| \leq \\
\sum_{k=1}^{n+1} \hspace{0.1 cm}  &\sup_{\ul{x} \in  V^n \times \Omega}|| D_k g_{\rho} (\ul{A}_{\sigma_i}(\ul{x})) || \cdot ||(\ul{A}_{\sigma_i})_{k,j}|| \leq \\
\sum_{k=1}^{n+1} \hspace{0.1 cm}  &\sup_{\ul{x} \in  V^n \times \Omega}|| D_k g_{\rho} (\ul{x}) ||\cdot ||(\ul{A}_{\sigma_i})_{k,j}|| \leq \\
\sum_{k=1}^{n+1} \hspace{0.1 cm}  &\sup_{\ul{x} \in  V^n \times \Omega}|| D_k g_{\rho} (\ul{x}) || \, ,
\end{split}
\end{equation}
where in the last step we have used the fact that every component of  $\ul{A}_{\sigma_i}$ is either some identity matrix or a zero-matrix, from which it follows that $||(\ul{A}_{\sigma_i})_{k,j}|| \leq 1$ for all $1 \leq k,j \leq n+1$. From the above we see that it is sufficient to prove that 
\begin{equation}\label{eq:4.26}
\lim_{\rho \downarrow 0} \sup_{\ul{x} \in V^n \times \R}  || D{g_{\rho}} (\ul{x})||  = 0 \, .
\end{equation}
The proof of this fact can directly be copied from \cite{vdbauw}, the only difference being that $g_{\rho}(\ul{x})$ does not map $V^n \times \Omega$ to itself. Nevertheless, we will reproduce it here for the sake of completeness. For all $\rho > 0$ we have

\begin{equation}
\begin{split}\label{eq:4.27}
\sup_{\ul{x} \in V^n \times \Omega} &|| D{g_{\rho}} (\ul{x})|| = 
\sup_{||\ul{x}|| \leq 2\rho} || D{g_{\rho}} (\ul{x})|| = \\
\sup_{||\ul{x}|| \leq 2\rho} &|| \chi(\rho^{-1}\ul{x})Dg(\ul{x}) + \rho^{-1}g(\ul{x})D\chi(\rho^{-1}\ul{x}) || \leq \\
\sup_{||\ul{x}|| \leq 2\rho} &|| \chi(\rho^{-1}\ul{x})|| \sup_{||\ul{x}|| \leq 2\rho} ||Dg(\ul{x}) || +\, \rho^{-1}\!\! \sup_{||\ul{x}|| \leq 2\rho} ||g(\ul{x})|| \sup_{||\ul{x}|| \leq 2\rho} ||D\chi(\rho^{-1}\ul{x}) || \leq \\
\sup_{||\ul{x}|| \leq 2\rho} &C_1 ||Dg(\ul{x}) || + \rho^{-1}C_2 \sup_{||\ul{x}|| \leq 2\rho}  ||g(\ul{x})|| \, ,
\end{split} 
\end{equation}
where we have set 
\begin{equation}\label{eq:4.28}
C_1:=\sup_{\ul{x} \in V^n \times \Omega} ||\chi(\ul{x}) ||\, , 
\end{equation}
and
\begin{equation}\label{eq:4.29}
C_2:=\sup_{\ul{x} \in V^n \times \Omega} ||D\chi(\ul{x}) ||\, .
\end{equation}
By the mean value theorem we have, whenever $||\ul{x}|| \leq 2\rho$,

\begin{equation}\label{eq:4.30}
||g(\ul{x})|| = ||g(\ul{x}) - g(0)|| \leq ||\ul{x}||  \sup_{||\ul{x}|| \leq 2\rho} ||Dg(\ul{x}) || \, .
\end{equation}
Combining inequalities  (\ref{eq:4.27}) and (\ref{eq:4.30}), we get
\begin{equation}\label{eq:4.31}
\begin{split}
&\sup_{\ul{x} \in V^n \times \Omega}|| D{g_{\rho}} (\ul{x})|| \leq \\
&\sup_{||\ul{x}|| \leq 2\rho} C_1 ||Dg(\ul{x}) || + \rho^{-1}C_2 \sup_{||\ul{x}|| \leq 2\rho} ||\ul{x}|| \sup_{||\ul{x}|| \leq 2\rho}||Dg(\ul{x}) || = \\
&\sup_{||\ul{x}|| \leq 2\rho} C_1 ||Dg(\ul{x}) || + \rho^{-1}C_2 \cdot 2 \rho  \sup_{||\ul{x}|| \leq 2\rho}||Dg(\ul{x}) || = \\
&(C_1 + 2C_2) \sup_{||\ul{x}|| \leq 2\rho}  ||Dg(\ul{x}) ||\, .
\end{split}
\end{equation}
Because $g(\ul{x})$ is at least $C^1$ and $Dg(0) = 0$, it follows that
\begin{equation}\label{eq:4.32}
\lim_{\rho \downarrow 0} \sup_{\ul{x} \in V^n \times \Omega} || D{g_{\rho}} (\ul{x})||  = 
(C_1 + 2C_2) \lim_{\rho \downarrow 0}  \sup_{||\ul{x}|| \leq 2\rho} ||Dg(\ul{x}) || = 0 \, .
\end{equation}
This proves the theorem.
\end{proof}

\noindent Theorem  \ref{T3} implies that the system 
\begin{equation}\label{eq:4.33}
\dot{\ul{x}}  = D\underline{\Gamma}_f(0)\ul{x} + \ul{\Gamma}_{g_{\rho}}(\ul{x}) 
\end{equation}
admits a global center manifold for small enough $\rho>0$. Recall that the vector field on the right hand side of (\ref{eq:4.33})  can be written as

\begin{equation}\label{eq:4.34}
D\underline{\Gamma}_f(0)\ul{x} + \ul{\Gamma}_{g_{\rho}}(\ul{x}) = \underline{\Gamma}_h(\ul{x}) + \ul{\Gamma}_{g_{\rho}}(\ul{x}) = \underline{\Gamma}_{h+g_{\rho}}(\ul{x}) \, ,
\end{equation}
where $h(\ul{x}) = Df(0)\ul{x}$. It follows that this vector field is again $\{\ul{A}_{\sigma_i}\}$-equivariant. Moreover, by theorem \ref{T2} it agrees with our initial vector field $\ul{\Gamma}_f$ on an open neighborhood around the origin. In the coming sections we shall investigate the properties of the center manifold of equation  (\ref{eq:4.33}).

\section{Symmetry and the center manifold}\label{sec4}

Recall that the global center manifold of an ODE at a steady state point contains all its bounded solutions, such as the steady state points and periodic orbits near the steady state. Therefore, when studying local bifurcations one is often only interested in the dynamics on this manifold. We will now show that the center manifold dynamics inherits the symmetries of the original fundamental network. Moreover, we show that every possible equivariant  vector field on the center manifold may arise after center manifold reduction. 


We begin  by fixing some notation. Given a smooth function $f: V^n\times \Omega \rightarrow V$, we may view the restriction $\ul{\Gamma}_f(\bullet,0)$ of $\ul{\Gamma}_f$ to $\{\la =0 \}$ as a function from $V^n$ to itself. Let us denote this function by  

\[ \Gamma_{f,0}: V^n \rightarrow V^n \, .\]
We will then write  

\[W_c, W_h \subset V^n\]
 for respectively the center subspace and the hyperbolic subspace  of $D\Gamma_{f,0}(0)$. Recall that we have
 \begin{equation}
 V^n = W_c \oplus W_h \, ,
 \end{equation}
and let us denote by $P_c$ and $P_h$ respectively the projection on $W_c$ and $W_h$ corresponding to this decomposition. 
The spaces $W_{c}$ and $W_h$ are invariant under the action of $\{A_{\sigma_i}\}_{i=1}^n$. More generally, given any differentiable  function $F: \R^n \rightarrow \R^n$ and linear function $B: \R^n \rightarrow \R^n$ such that

\begin{equation}\label{eq:5.1}
F \circ B = B \circ F \, ,
\end{equation}
we also have that

 \begin{equation}\label{eq:5.2}
DF(0) \circ B = B \circ DF(0) \, .
\end{equation}
From this it follows that $B$ maps the center and hyperbolic subspaces of $DF(0)$ into themselves.
The following theorem states that the dynamics of $\ul{\Gamma}_f$, restricted to its center manifold, is conjugate to a $\Sigma$-equivariant system on $W_c$.

\begin{thr}\label{T5}
Let $k\geq 1$ and let $f: V^n\times \Omega \rightarrow V$ be of class $C^k$. Assume that the vector field $\ul{\Gamma}_f(\ul{x})$ satisfies the conditions of theorem \ref{T1}, so that its center manifold $M_c$ exists. 
Then the projection $P:V^n \times \Omega \rightarrow W_c \times \Omega$, given by $P(x,\la) := (P_c(x),\la)$ has the property  that its restriction $P|_{M_c}$ bijectively conjugates $\ul{\Gamma}_f|_{M_c}$ to  an ODE on $W_c \times \Omega$ of the form
\begin{equation}\label{eq:5.3}
\begin{split}
\dot{x} &= R(x, \la) \, , \\
\dot{\la} &= 0 \, .
\end{split}
\end{equation}
Here, $R: W_c \times \Omega  \rightarrow W_c$ is a $C^k$-function satisfying
\begin{itemize}
\item $R(0,0) = 0$.
\item The center subspace of $DR_{0}(0)$ is the full space $W_c$, where we have set \\ $R_{0} = R(\bullet,0): W_c \rightarrow W_c$ as the restriction of $R$ to $\{\la = 0\}$.
\item  $R(A_{\sigma_i}x, \la) = A_{\sigma_i}R(x,\la)$ for all $i \in \{1,\dots n \}$ and $(x,\la) \in W_c \times \Omega$. Here $A_{\sigma_i}$ denotes the restriction of $A_{\sigma_i}$ to $W_c$.
\end{itemize}
\end{thr}
\noindent 
We will call the map $R: W_c \times \Omega  \rightarrow W_c$ of the preceding theorem the \textit{reduced vector field} of the network vector field $\ul{\Gamma}_f$. Note that the statement of theorem \ref{T5} is not that any vector field $R$ that satisfies the conclusions of theorem \ref{T5} can be obtained as the reduced vector field of a network vector field. This issue will be addressed in theorems \ref{linall} and \ref{nonlin} and remarks \ref{rem2} and \ref{rem3}, where it is shown that indeed theorem \ref{T5} exactly describes all possible reduced vector fields.

The result of theorem $\ref{T5}$ hinges mostly on a corollary of theorem \ref{T1}, which states that symmetries of a vector field  are passed on to its center manifold. More precisely, we have the following result.

\begin{lem}\label{L1}
Let $F$ be an  arbitrary vector field on $\R^n$ satisfying the conditions of theorem \ref{T1}. Keeping with the notation, let $\psi : X_c \rightarrow X_h$ be the map whose graph is the center manifold. Given a linear map $B: \R^n \rightarrow \R^n$ such that $F \circ B = B \circ F$, we also have $\psi \circ B = B \circ \psi$.  Furthermore, the center manifold is invariant under $B$, i.e. we have  $Bx\in M_c$ whenever $x \in M_c$.
\end{lem}

\begin{proof}
We will begin by showing the invariance of $M_c$. We remarked earlier that $DF(0)$ commutes with $B$ whenever $F$ does, and that both $X_c$ and $X_h$ are $B$-invariant spaces. From this it follows that the projections $\pi_c$ and $\pi_h$ with respect to the decomposition 

\begin{equation}\label{eq:5.4}
\R^n = X_c \oplus X_h \quad
\end{equation}
commute with $B$ as well.
Recall that the center manifold is given by

\begin{equation}\label{eq:5.5}
M_c = \{x \in \R^n :  \underaccent{t \in \R}{\text{sup }}||\pi_h\phi^t(x)|| <  \infty \} \, ,
\end{equation}
where $\phi^t(x)$ denotes the flow of $F$. Moreover, we have the following equality for the flow, valid for all $x \in \R^n$ and $t \in \R$,

\begin{equation}\label{eq:5.6}
\phi^t(Bx) = B\phi^t(x) \, .
\end{equation}
This follows directly from the symmetry of $F$. Now suppose $x \in M_c$. Then
\begin{equation}\label{eq:5.7}
\begin{split}
&||\pi_h\phi^t(Bx)|| = \\
&||\pi_hB\phi^t(x)|| = \\
&||B\pi_h\phi^t(x)|| \leq \\
&||B||||\pi_h\phi^t(x)||  \, .
\end{split}
\end{equation}
From this it follows that  
\begin{equation}\label{eq:5.8}
\underaccent{t \in \R}{\text{sup }}||\pi_h\phi^t(Bx)|| < \infty \, .
\end{equation}
Hence, $Bx$ is an element of $M_c$ as well.

To show that $\psi \circ B = B \circ \psi$,   note that the center manifold is (also) given by
\begin{equation}\label{eq:5.9}
M_c = \{ x_c + \psi(x_c) : x_c \in X_c \} \, .
\end{equation}
In other words, given an element $x \in M_c$ we may write $x = x_c + \psi(x_c)$ for some $x_c = \pi_c(x) \in X_c$. From this we see that
\begin{equation}\label{eq:5.10}
\pi_h(x) = \psi(x_c) = \psi(\pi_c(x)) \, .
\end{equation}
Now given $x_c \in X_c$, we know that $x_c + \psi(x_c)$ is an element of $M_c$. Hence by our first result, so is $Bx_c + B\psi(x_c)$. Applying equation (\ref{eq:5.10}) to the latter  gives 
\begin{equation}\label{eq:5.11}
\pi_h(Bx_c + B\psi(x_c))  =\psi(\pi_c(Bx_c + B\psi(x_c))) \, .
\end{equation}
Hence, since $\psi(x_c)$ is an element of $X_h$ and $B$ leaves both $X_c$ and $X_h$ invariant, equation (\ref{eq:5.11}) reduces to
\begin{equation}\label{eq:5.12}
 B\psi(x_c)  =\psi(Bx_c) \, .
\end{equation}
This proves the equivariance of $\psi$.
\end{proof}
\noindent Since we want to apply center manifold reduction to $\ul{\Gamma}_f$, let us denote by 
 \[\ul{W}_c, \ul{W}_h \subset V^n \times \Omega\] 
The center and hyperbolic subspaces of $D\ul{\Gamma}_f(0)$. These spaces are invariant under the action of $\{\ul{A}_{\sigma_i} \}_{i=1}^n$, as follows from the equivariance of $D\ul{\Gamma}_f(0)$. We   write $\ul{P}_c$ and $\ul{P}_h$ for the projections corresponding to

 \begin{equation}
 V^n \times \Omega = \ul{W}_c \oplus \ul{W}_h \, .
 \end{equation} 
The following lemma relates the center and hyperbolic subspaces of $D\ul{\Gamma}_f(0)$  to those of $D\Gamma_{f,0}(0)$.

\begin{lem}\label{L2}
The space $\ul{W}_h$ satisfies

\begin{equation}\label{eq:5.13}
\ul{W}_h = ({W}_h,0) \subset V^n \times \Omega \, .
\end{equation}
Furthermore, setting $l := \dim{\Omega}$, there exist vectors $\{w_i \}_{i=1}^{l}$ in $W_h$ and a basis $\{\la_i \}_{i=1}^{l}$ for $\Omega$, such that the vectors $\ul{w}_i := (w_i,\la_i) \in V^n \times \Omega$ satisfy
\begin{equation}\label{eq:5.14}
\ul{W}_c = ({W}_c,0) \oplus \spn\{\ul{w}_i:i = 1, \dots l\} \, .
\end{equation}
\end{lem}

\begin{proof}
By definition of $\ul{\Gamma}_f$, we see that its linearisation is of the form

\begin{equation}\label{eq:5.15}
D\ul{\Gamma}_f(0) =  \begin{pmatrix}
D\Gamma_{f,0}(0) & v \\
 0 & 0
\end{pmatrix}  \, ,
\end{equation}
corresponding to the natural decomposition of $V^n \times \Omega$. Here,  $v$ is a linear map from $\Omega$ to $V^n$ that is of no further importance to us. Now suppose that $v_{\kappa} \in V^n$ is a generalized eigenvector of  $D\Gamma_{f,0}(0)$ corresponding to the eigenvalue $\kappa \in \R$. It can then be seen from the above matrix  that $(v_\kappa, 0) \in V^n \times \Omega$ is a generalized eigenvector of $D\ul{\Gamma}_f(0)$ corresponding to the same eigenvalue. This likewise holds for complex eigenvalues. In particular, we conclude that
\begin{equation}\label{eq:5.16}
({W}_{c/h},0)  \subset  \ul{W}_{c/h} \, .
\end{equation}
Next, we note that the spectrum of $D\ul{\Gamma}_f(0)$ can be obtained from that of $D{\Gamma}_{f,0}(0)$ by $l$ times adding the eigenvalue $0$. Since $0$ is a purely imaginary number, it follows that we have in fact 
\begin{equation}\label{eq:5.17}
({W}_{h},0)  =  \ul{W}_{h} \, .
\end{equation}
Moreover, we see that there exist vectors $\{\ul{w}'_i\}_{i=1}^l$ in $V^n \times \Omega$ such that

\begin{equation}\label{eq:5.18}
\ul{W}_c = ({W}_c,0) \oplus \spn\{\ul{w}'_i:i = 1, \dots l\} \, .
\end{equation}
In fact, these $\ul{w}'_i$ are generalised eigenvectors of $D\ul{\Gamma}_f(0)$ for the eigenvalue $0$ that are not in $(W_c,0)$. Writing $\ul{w}'_{i}= (w_{i,c} + w_{i,h},\la_i)$ for $w_{i,c/h} \in W_{c/h}$ and $\la_i \in \Omega$, and noting that 
\begin{equation}\label{eq:5.19}
V^n \times \Omega = \ul{W}_c \oplus \ul{W}_h \quad,
\end{equation}
we may conclude that $\{\la_i \}_{i=1}^{l}$ forms a basis for $\Omega$. If we now set $\ul{w}_{i} := (w_{i,h},\la_i) =: (w_i, \la_i)$, we see that indeed 

\begin{equation}\label{eq:5.20}
\begin{aligned}
 &({W}_c,0) \oplus \spn\{\ul{w}_i:i = 1, \dots l\} = \\
 &({W}_c,0) \oplus \spn\{\ul{w}'_i:i = 1, \dots l\} = \ul{W}_c \, .
\end{aligned}
\end{equation}
This proves the lemma.
\end{proof}

\noindent We are now in the position to prove theorem \ref{T5}. In any center manifold, the projection  $\pi_c: M_c \subset \R^n  \rightarrow X_c$ gives rise to a conjugate system on the center subspace $X_c$. However, since the space $\ul{W}_c$ is in general not equal to $W_c \times \Omega$, some more work has to be done.  Recall that we denote by $P_c$ and $P_h$ respectively the projections on $W_c$ and $W_h$, corresponding to the decomposition

\begin{equation}\label{eq:5.21}
 V^n = W_c \oplus W_h \, .
 \end{equation}
Likewise, we denoted by $\ul{P}_c$ and $\ul{P}_h$ the projections on $\ul{W}_c$ and $\ul{W}_h$ for
\begin{equation}\label{eq:5.22}
 V^n \times \Omega = \ul{W}_c \oplus \ul{W}_h \, .
 \end{equation}
Because the spaces $W_{c/h}$ and $\ul{W}_{c/h}$ are $\{A_{\sigma_i}\}$ and $\{\ul{A}_{\sigma_i}\}$-invariant, respectively, it follows that $P_{c/h}$ and $\ul{P}_{c/h}$ commute with  $A_{\sigma_i}$ respectively $\ul{A}_{\sigma_i}$.

\begin{proof}[Proof of theorem \ref{T5}]

We begin by constructing a vector field on $\ul{W}_c$ conjugate to $\ul{\Gamma}_f|_{M_c}$, satisfying an analogue of the three bullet points in theorem \ref{T5}. From it, we then construct  the required vector field on $W_c \times \Omega$. 

It is clear that the projection $\ul{P}_c|_{M_c}:M_c \rightarrow \ul{W}_c$ defines a global chart for the manifold $M_c$. Hence, by taking the pushforward of $\ul{\Gamma}_f|_{M_c}$ we get a $C^k$-vector field $R_1$ on $\ul{W}_c$ defined by
\begin{equation}\label{eq:5.23}
R_1(\ul{x}_c) = \ul{P}_c \ul{\Gamma}_f(\ul{x}_c + \psi(\ul{x}_c))  \  \text{ for } \ul{x}_c \in \ul{W}_c \, .
\end{equation}
We note that it has the following properties:  First of all, because $\psi(0) = 0$ and $\ul{\Gamma}_f(0) = 0$, we see that
\begin{equation}\label{eq:5.24}
R_1(0) = \ul{P}_c \ul{\Gamma}_f( 0) = 0 \, .
\end{equation}
Next, the derivative of $R_1$ at the origin satisfies

\begin{equation}\label{eq:5.25}
DR_1(0)v = \ul{P}_c D\ul{\Gamma}_f(0)(v + D\psi(0)v) = \ul{P}_c D\ul{\Gamma}_f(0)v 
\end{equation}
for all $v \in \ul{W}_c$, where we have used that $D\psi(0) = 0$. Hence, we have the identity $DR_1(0) = \ul{P}_c D\ul{\Gamma}_f(0)|_{\ul{W}_c}$. From this it follows that the spectrum of $DR_1(0)$ lies entirely on the imaginary axis.
Finally, the vector field $R_1$ shares the symmetries of $\ul{\Gamma}_f$. Indeed, by using lemma \ref{L1} we get
\begin{equation}\label{eq:5.26}
\begin{split}
R_1(\ul{A}_{\sigma_i}\ul{x}_c) = \hspace{0.1 cm}&\ul{P}_c \ul{\Gamma}_f(\ul{A}_{\sigma_i}\ul{x}_c + \psi(\ul{A}_{\sigma_i}\ul{x}_c)) = \\
&\ul{P}_c \ul{\Gamma}_f(\ul{A}_{\sigma_i}\ul{x}_c + \ul{A}_{\sigma_i}\psi(\ul{x}_c)) = \\
&\ul{P}_c \ul{A}_{\sigma_i}\ul{\Gamma}_f(\ul{x}_c + \psi(\ul{x}_c)) = \\
& \ul{A}_{\sigma_i}\ul{P}_c\ul{\Gamma}_f(\ul{x}_c + \psi(\ul{x}_c)) = \\
& \ul{A}_{\sigma_i}R_1(\ul{x}_c) \, ,
\end{split}
\end{equation}
for all $i \in \{1, \dots n\}$ and $\ul{x}_c \in \ul{W}_c$. 

Next, we define the linear map 
\begin{equation}\label{eq:5.27}
\begin{aligned}
P': \ul{W}_c \rightarrow W_c \times \Omega\, , \ 
(x,\la) \mapsto (P_c(x), \la) \, .
\end{aligned}
\end{equation}
By lemma \ref{L2}, we know that the space $\ul{W}_c$ can be written as

\begin{equation}\label{eq:5.28}
\ul{W}_c = ({W}_c,0) \oplus \spn\{\ul{w}_i:i = 1, \dots l\} \, ,
\end{equation}
for vectors $\ul{w}_i = (w_i, \la_i)$ with $w_i \in W_h$ and $\{ \la_i\}_{i=1}^l$ a basis for $\Omega$. Because $P'$ is the identity on $(W_c , 0)$ and sends the elements $\ul{w}_i = (w_i, \la_i)$ to $(0,\la_i)$, we conclude that it is a bijection. Furthermore, the  map $P'$ is $\{\ul{A}_{\sigma_i}\}$-equivariant, as 
\begin{equation}\label{eq:5.29}
\begin{aligned}
P' \circ \ul{A}_{\sigma_i}(x,\la) &=&&P'(A_{\sigma_i}x,\la) &=&& (P_c ( A_{\sigma_i}x),\la) &= \\
 & &&( A_{\sigma_i}  P_c (x),\la) &=&& \ul{A}_{\sigma_i} (P_c (x),\la)   &=     \ul{A}_{\sigma_i} \circ P'(x,\la) \, ,
\end{aligned}
\end{equation}
for all $i \in \{1, \dots n\}$ and $(x,\la) \in \ul{W}_c$. Note that this also implies the $\{\ul{A}_{\sigma_i}\}$-invariance of $W_c \times \Omega$. Taking the pushforward of $R_1$ under $P'$ now yields a $C^k$-vector field $R_2$ on $W_c \times \Omega$ given by
\begin{equation}\label{eq:5.30}
R_2(\ul{x}) = P' \circ R_1 \circ P'^{-1}(\ul{x}) = P' \circ \ul{P}_c \circ \ul{\Gamma}_f[P'^{-1}(\ul{x}) + \psi(P'^{-1}(\ul{x})) ]\, ,
\end{equation}
for $\ul{x}$ in $W_c \times \Omega$. From the properties of $R_1$ it follows that $R_2$ maps $0$ to $0$, that $DR_2(0)$ has a purely imaginary spectrum and that $R_2$ is $\{\ul{A}_{\sigma_i}\}$-equivariant.  

Finally, we want to show that the conjugacy $P := P' \circ \ul{P}_c: V^n \times \Omega \rightarrow W_c \times \Omega$ is as stated in theorem \ref{T5}, i.e. that $P(x,\la) = (P_c(x), \la)$. However, we know  that $\ul{P}_c$ vanishes on $\ul{W}_h = (W_h,0)$, hence so does $P$. Likewise, we may conclude that $P$ is the identity on $(W_c,0)$, as both $\ul{P}_c$ and $P'$ are. Moreover, for any of the elements $\ul{w}_i := (w_i, \la_i) \in \ul{W}_c$ we have that $P(w_i, \la_i) = P'(w_i, \la_i) = (0,\la_i)$, where we have used that $w_i \in W_h$. This proves that $P$ is indeed of the required form. In particular, since it is the identity on the $\Omega$-component, and since $\ul{\Gamma}_f(\ul{x})$ has $\Omega$-component $0$, we conclude that 

\begin{equation}\label{eq:5.31}
\begin{aligned}
R_2(\ul{x}) = & P' \circ \ul{P}_c \circ \ul{\Gamma}_f[ P'^{-1}(\ul{x}) + \psi(P'^{-1}(\ul{x})) ]= \\ & P \circ \ul{\Gamma}_f[ P'^{-1}(\ul{x}) + \psi(P'^{-1}(\ul{x})) ] 
\end{aligned}
\end{equation}
has vanishing $\Omega$-component as well. Therefore, we may write it as
\begin{equation}\label{eq:5.32}
\begin{split}
R_2(x,\la)&= (R(x, \la) ,0) \, ,
\end{split}
\end{equation}
and it follows from the properties of $R_2$ that $R(0,0) = 0$, that $DR_0(0)$ has a purely imaginary spectrum, where we have set $R_0 := R(\bullet,0)$, and that $R_0$ is $\{A_{\sigma_i}\}$-equivariant. This proves the theorem.
\end{proof}
\noindent Next we want to describe all the reduced vector fields $R$ that can be obtained after center manifold reduction in a fundamental network vector field through the procedure of theorem \ref{T5}. We start with the linear part of $R$.

\begin{thr}\label{linall}
Let $\ul{\Gamma}_f$ be a fundamental network vector field satisfying the conditions of theorem \ref{T1} and let $R: W_c \times \Omega \rightarrow W_c$ be its corresponding reduced vector field. Then the linear part of $R$ is given explicitly by 

\begin{equation}\label{eq:5.00k}
\begin{aligned}
&D_{x}R(0,0) = D\Gamma_{f,0}(0)|_{W_c}\\
&D_{\la}R(0,0) = P_c \circ D_{\la}{\Gamma}_f(0,0)\,.
\end{aligned}
\end{equation}
Moreover, let 

\begin{equation}\label{eq:5.01k}
V^n = W_1 \oplus W_2
\end{equation}
be any decomposition of $V^n$ into $\{A_{\sigma_i} \}$-invariant spaces and suppose that we are given a linear map

\begin{equation}\label{eq:5.02k}
\tilde{R}:W_1 \times \Omega \rightarrow W_1
\end{equation}
such that $\tilde{R}|_{W_1\times \{0\}}$ has a purely imaginary spectrum. Assume furthermore that $\tilde{R}$ intertwines the action of $\{\ul{A}_{\sigma_i}\}$ on $W_1 \times \Omega$ with that of $\{{A}_{\sigma_i}\}$ on $W_1$, i.e. that $R(A_{\sigma_{i}}x_c,\la) =  A_{\sigma_i}R(x_c,\la)$ for all $(x_c,\la) \in W_c \times \Omega$ and all $\sigma_i \in \Sigma$. Then there exists a fundamental network vector field $\ul{\Gamma}_g$ such that the center and hyperbolic subspaces of $D{\Gamma}_{g,0}(0)$ are equal to $W_1$ respectively $W_2$ and such that (the linear part of) its reduced vector field is equal to $\tilde{R}$.
\end{thr}

\begin{proof}
Recall from the proof of theorem \ref{T5} that we have 

\begin{equation}\label{eq:5.03k}
DR_2(0,0) = P' \circ \ul{P}_c \circ D\ul{\Gamma}_f(0)|_{\ul{W}_c} \circ  P'^{-1}\, ,
\end{equation}
where $R_2= (R,0): W_c \times \Omega \rightarrow W_c \times \Omega$ and where $P': \ul{W}_c \rightarrow W_c \times \Omega$ is given by $P'(x,\la) = (P_c(x), \la)$. The linear map $P'$ is the identity on $(W_c,0)$ and sends the elements $(w_i,\la_i) \in \ul{W}_c$ from lemma \ref{L2} to $(0,\la_i)$, from which it follows that we may write 

\begin{equation}\label{eq:5.04k}
P'^{-1}(x_c,\la) = (x_c + Q(\la),\la ) 
\end{equation}
for a linear map $Q: \Omega \rightarrow W_h$. Explicitly this map is given by

\begin{equation}\label{eq:5.05k}
Q(\la_i) = w_i
\end{equation}
for $(w_i,\la_i)$ an element as described in lemma \ref{L2} and where we use that $\{\la_i\}_{i=1}^l$ forms a basis for $\Omega$. From this we see that

\begin{equation}\label{eq:5.06k}
\begin{aligned}
&DR_2(0,0)(x_c,\la) = \\
&P' \circ \ul{P}_c \circ D\ul{\Gamma}_f(0)(x_c + Q(\la),\la) = \\
&P'[\,(\,D{\Gamma}_{f,0}(0)(x_c + Q(\la)) + D_{\la}{\Gamma}_f(0,0)(\la),0\,)\,] = \\
&(\,P_c[D{\Gamma}_{f,0}(0)(x_c + Q(\la)) + D_{\la}{\Gamma}_f(0,0)(\la)],0\,) = \\
&(\,D{\Gamma}_{f,0}(0)(x_c) + P_c \circ D_{\la}{\Gamma}_f(0,0)(\la),0\,)\, ,
\end{aligned}
\end{equation}
where in the second step we have used that the linearisation of $\ul{\Gamma}_f$ is given by
\begin{equation}
D\ul{\Gamma}_f(0) =  \begin{pmatrix}
D\Gamma_{f,0}(0) & D_{\la}{\Gamma}_f(0,0) \\
 0 & 0
\end{pmatrix}  \, .
\end{equation}
As $R$ is defined by $R_2= (R,0)$, we see that indeed 

\begin{equation}\label{eq:5.07k}
\begin{aligned}
&D_{x}R(0,0) = D\Gamma_{f,0}(0)|_{W_c}\\
&D_{\la}R(0,0) = P_c \circ D_{\la}{\Gamma}_f(0,0)\,.
\end{aligned}
\end{equation}
This proves the first part of the theorem.

As for the second part, if $W_1$, $W_2$ and $\tilde{R}$ are given as in the statement of the theorem, then we may define a  linear vector field on $V^n \times \Omega = W_1 \oplus W_2 \oplus \Omega$ by the matrix

\begin{equation}
A :=  \begin{pmatrix}
\tilde{R}|_{W_1} & 0 & \tilde{R}|_{\Omega} \\
 0  & {(-) \Id_{W_2}} & 0 \\
 0 & 0 & 0 
\end{pmatrix}  \, .
\end{equation}
We claim that $A$ is a $\la$-family of fundamental network vector fields. Indeed, it follows from the invariance of $W_1$ and $W_2$ and from the equivariance of $\tilde{R}$ that $A$ commutes with $\ul{A}_{\sigma_i} = (A_{\sigma_i}, \Id_{\Omega})$ for all $\sigma_i \in \Sigma$. Note in particular that this implies that the map

\begin{equation}
v :=  \begin{pmatrix}
\tilde{R}|_{\Omega}\\
0
\end{pmatrix}  
: \Omega \rightarrow V^n
\end{equation}
from the right hand corner of $A$ satisfies $A_{\sigma_i} v = v$. From this it follows that $v_{\sigma_i} = (A_{\sigma_i} v )_{\sigma_1}= v_{\sigma_1}$ for all $\sigma_i \in \Sigma$, where $\sigma_1$ denotes the unit in $\Sigma$. Hence the $n$ components of $v(\la) \in V^n$ are all equal. This latter fact is necessarily the case for a $\la$-family of fundamental network vector fields, since it is only the response function and not the network structure that depends on $\la$. We will therefore write $A = \ul{\Gamma}_g$. 

It is clear that $(-)\Id_{W_2}$ has a purely hyperbolic spectrum, whereas $\tilde{R}|_{W_1}$ is given to have only eigenvalues on the imaginary axis. Hence we conclude that the center and hyperbolic subspaces of $D{\Gamma}_{g,0}(0)$ are equal to $W_1$ respectively $W_2$. Furthermore, it follows from the first part of the theorem that the reduced vector field of $A = \ul{\Gamma}_g$ is indeed equal to $\tilde{R}$. This concludes the proof.
\end{proof}

\noindent Theorem \ref{linall} tells us that any linear map satisfying the bullet points of theorem \ref{T5} can occur as the linear part of the reduced vector field of a fundamental network vector field. The following result tells us that furthermore any equivariant nonlinear part can be realised in a reduced vector field.

\begin{thr}\label{nonlin}
Let $A: V^n \times \Omega \rightarrow V^n \times \Omega$ be a (fixed) linear fundamental network vector field and let $\ul{W}_c$, $\ul{W}_h$,  $W_c$ and $W_h$ be the invariant spaces determined by $A$ and $A|_{V^n}$. It follows from theorem \ref{linall} that the linear part of the reduced vector field $R$ of a fundamental network vector field $\ul{\Gamma}_f$ is completely determined by the linear part of $\ul{\Gamma}_f$. In particular, if $D\ul{\Gamma}_f(0) = A$ then we will denote the linear part of $R$ by 

\begin{equation}
\tilde{A}:= DR(0): W_c \times \Omega \rightarrow W_c\, .
\end{equation} 
Let $G: W_c \times \Omega \rightarrow W_c$ be a $C^1$ map satisfying $G(0) = 0$ and $DG(0) = 0$ and assume furthermore that $G \circ \ul{A}_{\sigma_i} = A_{\sigma_i} \circ G$ for all $\sigma_i \in \Sigma$. Then there exists a fundamental network vector field $\ul{\Gamma}_f$ with linear part $A$ satisfying the conditions of theorem $\ref{T1}$ and with reduced vector field given locally by $\tilde{A} + G$.

\end{thr}

\begin{proof}
Given $G: W_c \times \Omega \rightarrow W_c$ we may define the vector field $(G,0)$ on $W_c \times \Omega$ by 

\begin{equation}
(G,0)(x_c,\la) := (G(x_c,\la),0)\, .
\end{equation} 
 Next, we define the vector field on  $\ul{W_c}$ given by

\begin{equation}
\tilde{G}:= P'^{-1}\circ (G,0) \circ P' \, .
\end{equation} 
Note that $\tilde{G}$ is $\{\ul{A}_{\sigma_i}\}$-equivariant by construction and satisfies $\tilde{G}(0) = 0$ and $D\tilde{G}(0) = 0$. We furthermore see that $\tilde{G}$ has vanishing $\la$-component, as this is the case for $(G,0)$ and because $P'$ respects the $\la$-component. These properties likewise hold for the vector field $(\tilde{G},0)$ on $\ul{W_c} \oplus \ul{W_h} = V^n \times \Omega$, from which it follows that the vector field $A+(\tilde{G},0)$ is a fundamental network vector field with linear part $A$. 

 Finally, let $\ul{\Gamma}_f$ be a fundamental network vector field satisfying the conditions of theorem \ref{T1} and agreeing locally with $A+(\tilde{G},0)$ around the origin. The dynamics on the center manifold of $\ul{\Gamma}_f$ is then  conjugate to

\begin{equation}
\begin{aligned}
&\ul{P}_c\ul{\Gamma}_f(\ul{x}_c + \Psi(\ul{x}_c)) = \ul{P}[A+(\tilde{G},0)](\ul{x}_c + \Psi(\ul{x}_c)) = \\
&\ul{P}_cA(\ul{x}_c + \Psi(\ul{x}_c))+ \tilde{G}(\ul{x}_c) = A(\ul{x}_c)+ \tilde{G}(\ul{x}_c) = \\
 &(A+\tilde{G})(\ul{x}_c)
\end{aligned}
\end{equation}
for $\ul{x}_c \in \ul{W}_c$ sufficiently close to the origin. Conjugating by $P'$ we get the system

\begin{equation}
\dot{x} = (\tilde{A}x + G(x),0)
\end{equation} 
on a neighborhood around the origin in $W_c \times \Omega$, from which we conclude that the reduced vector field of $\ul{\Gamma}_f$ is indeed given locally by $\tilde{A}+G$. This proves the theorem.
\end{proof}

\begin{percent}

The reduced vector field $R$ of $\ul{\Gamma}_f$ is given by the $W_c$-part of the vector field 
\begin{equation}
R_2(\ul{x})  = P' \circ \ul{\Gamma}_f[ P'^{-1}(\ul{x}) + \psi(P'^{-1}(\ul{x}))] \ \mbox{on}\  W_c \times \Omega\, ,
\end{equation}
the $\Omega$-part of $R_2$ being zero. Here, $P': \ul{W}_c\to W_c\times\Omega$ was defined by $P':(x,\lambda) \mapsto (P_c(x),\lambda)$. From this it follows that 
\begin{equation}
DR_2(0)  = P' D\ul{\Gamma}_f(0) P'^{-1} \, ,
\end{equation}
because $D\psi(0)=0$. We note that $DR_2(0)|_{(W_c,0)} = (D\ul{\Gamma}_{f,0}(0)|_{W_c},0)$. Furthermore, writing 
\begin{equation}
\ul{W}_c = ({W}_c,0) \oplus \spn\{\ul{w}_i:i = 1, \dots l\} \quad
\end{equation}
for vectors $\ul{w}_i = (w_i, \la_i)$ with $w_i \in W_h$ and $\{ \la_i\}_{i=1}^n$ a basis for $\Omega$, we have that 
\begin{equation}
\begin{split}
DR_2(0)(0, \la_i) &= P D\ul{\Gamma}_f(0) P'^{-1}(0,\la_i) = P D\ul{\Gamma}_f(0) (w_i,\la_i) \\
&= P (D\ul{\Gamma}_{f,0}(0)w_i + v\la_i,0) = (P_cv\la_i,0) \, .
\end{split}
\end{equation}

\noindent We conclude that the linearisation of the reduced vector field is of the form
 
\begin{equation}
DR(0)(x_c,\la) = A_c x_c + v_c \la \, ,
\end{equation}
for $x_c \in W_c$, $\la \in \Omega$, and where $A_c := D\ul{\Gamma}_{f,0}(0)$ is $\{A_{\sigma_i}\}$-equivariant and $v_c=P_cv: \Omega\to W_c$ is such that $A_{\sigma_i}v_c=v_c$ for all $i=1, \ldots, n$.

Suppose now that we are given an arbitrary $\{A_{\sigma_i}\}$-equivariant linear map $A_c: W_c \rightarrow W_c$ with strictly imaginary spectrum and an arbitrary linear map $v_c: \Omega \rightarrow W_c$ such that $A_{\sigma_i}v_c = v_c$ for all $i = 1, \ldots, n$. We may then construct the linear map $A : W_c \times W_h \times \Omega \rightarrow W_c \times W_h \times \Omega$ given by
\begin{equation}
A =  \begin{pmatrix}
A_c & 0 & v_c \\
 0  & {\rm Id} & 0 \\
 0 & 0 & 0 
\end{pmatrix}  \, .
\end{equation}
It follows from its symmetry that $A$ is a $\la$-indexed family of network vector fields. Furthermore, the center subspace of $A$ equals $W_c$ and the reduced vector field of  $A$ is exactly given by   \begin{equation}
R(x_c,\la) = A_c x_c + v_c\la \, .
\end{equation}

\noindent This shows that any linear $R$ satisfying the conditions of theorem \ref{T5} can be obtained as a reduced vector field.

We now prove the same for nonlinear $R$. For this purpose, we let $A: V^n \times \Omega \to V^n \times \Omega$ be a fixed $\{\ul{A}_{\sigma_i}\}$-equivariant linear map with fixed center and hyperbolic subspaces $\ul{W}_{c/h}$, and such that every element in its image has vanishing $\Omega$-component. We will show that any nonlinear equivariant map $G_c: W_c\times \Omega \to W_c$ can arise as the nonlinear part of $R$. Indeed, assume that $G_{c}(0,0) = 0$ and $DG_{c}(0,0) = 0$ and that $G_c$ is $\{A_{\sigma_i}\}$-invariant, but arbitrary otherwise. Define the map
$F: W_c\times W_h\times \Omega \to W_c\times W_h\times \Omega$ by
$$F(x_c, x_h, \lambda) := A(x, \lambda) +(G_c(x_c,\lambda), 0, 0)\, .$$
It is clear that $F$ is $\{ \ul{A}_{\sigma_i}\}$-equivariant and hence has the structure of a homogeneous coupled cell network. A center manifold of $F$ is the linear subspace $\ul{W}_{c}$, and a computation shows that the reduced vector field of $F$ is equal to 
$$R(x_c, \lambda)=P' A|_{\ul{W}_c} (x_c, 0, \lambda) + G_c(x_c, \lambda)$$
as required.

\noindent Indeed, suppose that we are given an arbitrary $\{\ul{A}_{\sigma_i}\}$-equivariant  vector field

\begin{equation}\label{eq:5.37}
\begin{split}
&G_c: \ul{W}_c \rightarrow (W_c,0) \subset  \ul{W}_c\, , 
\end{split}
\end{equation}
satisfying $G_{c/h}(0,0) = 0$ and $DG_{c/h}(0,0) = 0$. It is then clear that the product
\[F_{\oplus}:= A + G_c \oplus G_h = (A|_{\ul{W}_c} + G_c) \oplus  (A|_{\ul{W}_h} + G_h): V^n \times \Omega \rightarrow V^n \subset V^n \times \Omega\]
is $\{\ul{A}_{\sigma_i}\}$-equivariant  as well. Furthermore, because it has vanishing $\Omega$-component, we may conclude by theorem \ref{symfun} that it is a $\la$-indexed family of network vector fields.  Now, although this system may not have a unique center manifold, it can immediately be seen that the linear space $\ul{W}_c$ itself is a center manifold. This is because the $\ul{W}_h$-component of $F_{\oplus}$ does not depend on $\ul{W}_c$ (and vice versa). In particular, the dynamics on this center manifold is just given by $A|_{\ul{W}_c} + G_c$. 

We know from the preceding section that there exists a network vector field $\widetilde{F}_{\oplus}$ that agrees with $F_{\oplus}$ in a neighbourhood around the origin, yet which does admit a unique center manifold. Some care is in order though, as $\widetilde{F}_{\oplus}$ may not be in the same form as $F_{\oplus}$ anymore. Specifically, its $\ul{W}_h$-component may depend on the $\ul{W}_c$-component of its inputs, from which it follows that the center manifold may not be equal to $\ul{W}_c$ anymore.  Nevertheless, sufficiently close to the origin the dynamics on the center manifold is still conjugate to $A|_{\ul{W}_c} + G_c$. Indeed, taking the push forward of $\widetilde{F}_{\oplus}|_{M_c}$ by the projection $ \ul{P}_c$ gives 
\begin{equation}
\begin{split}
\ul{P}_c\widetilde{F}_{\oplus}(\ul{x}_c + \psi(\ul{x}_c)) &= \ul{P}_c[(A|_{\ul{W}_c} + G_c)(\ul{x}_c)+(A|_{\ul{W}_h} + G_h)(\psi(\ul{x}_c) )] \\
&= (A|_{\ul{W}_c} + G_c)(\ul{x}_c) \, ,
\end{split}
\end{equation}
for $\ul{x}_c\in \ul{W}_c$ sufficiently close to the origin. Thus, we can construct a vector field on $\ul{W}_c$ that agrees (globally or locally) with $A|_{\ul{W}_c} + G_c$. Following the proof of theorem \ref{T5}, we use the fact that conjugating by the map $P': (x,\la) \mapsto (P_c(x),\la)$ defines a bijection between 
\[ \{ \ul{A}_{\sigma_i}\text{-invariant vector fields on }\ul{W}_c\text{ with vanishing }\Omega\text{-component}\}\]
on the one side, and
\[ \{\ul{A}_{\sigma_i}\text{-invariant vector fields on }W_c \times \Omega\text{ with vanishing }\Omega\text{-component}\}\]
 on the other. Hence, any $\{\ul{A}_{\sigma_i}\}$-equivariant function $G'_c: W_c \times \Omega \rightarrow (W_c, 0) \subset W_c \times \Omega$ satisfying $G'_c(0) = 0$ and $DG'_c(0) = 0$ is the nonlinear part of a reduced vector field for a network vector field with linearisation $A$.

\end{percent}

  \noindent Combining theorems \ref{linall} and \ref{nonlin} we see that any vector field $R$ satisfying the bullet points of theorem \ref{T5} can be achieved as the reduced vector field of some fundamental network vector field $\ul{\Gamma}_f$. The linear part of $R$ is completely determined by the linear part of $\ul{\Gamma}_f$, and we see that the nonlinear part of $R$ can  be any equivariant map on $W_c$ (which is determined once the linear part of $\ul{\Gamma}_f$ is fixed). This last observation will be important in the following remark.

\begin{remk}\label{rem2}
It is well known that a center manifold for the general ODE $\dot{x} = F(x)$ satisfies the tangency equation 
\begin{equation}\label{eq:5.38}
D\psi(x_c) \cdot \pi_c F(x_c + \psi(x_c)) = \pi_h F(x_c + \psi(x_c)) \, ,
\end{equation}
for $x_c \in X_c$, and where the graph of $\psi:X_c \rightarrow X_h$ equals the center manifold. In particular, keeping $\pi_c$ and $\pi_h$, that is $X_c$ and $X_h$ fixed, one can use this formula to express any Taylor coefficient of $\psi$ around $0$ as a rational function of a finite set of Taylor coefficients  of $F$ around $0$. See for example \cite{vdbauw} or \cite{wim}. This phenomenon is known as \textit{finite determinacy}. 

Returning to the setting of networks, if $D\ul{\Gamma}_f(0)$ and therefore $W_{c/h}$ and $\ul{W}_{c/h}$ are fixed, then the Taylor coefficients of the vector field on $\ul{W}_c$,
\begin{equation}\label{eq:5.39}
R_1(\ul{x}_c) = \ul{P}_c \ul{\Gamma}_f(\ul{x}_c + \psi(\ul{x}_c)) \, ,
\end{equation}
as well as those of 
\begin{equation}\label{eq:5.40} 
R_2(\ul{x}) = P' \circ R_1 \circ P'^{-1}(\ul{x})  \ \ \mbox{for}\ \ul{x} \in {W}_c \times \Omega\, ,
\end{equation}
are given by rational functions of the Taylor coefficients of $\ul{\Gamma}_f$. Combined with theorems \ref{nonlin} and \ref{linall} that state that any reduced vector field can be realised at least locally,
we may conclude that if some rational function of the Taylor coefficients of either of the two vector fields (\ref{eq:5.39}) or (\ref{eq:5.40}) is not forced zero by the symmetry, then it will in general not vanish. More precisely, such a rational function vanishing will be equivalent to some rational function of the coefficients of $\ul{\Gamma}_f$ vanishing. Note that to verify the occurrence of some bifurcation, one often needs to check that some rational functions of the Taylor coefficients of the vector field do not vanish. Therefore, center manifold reduction allows us to determine generic bifurcations in network vector fields. Of course, to verify whether such a bifurcation really occurs in a particular network, one actually has to compute and evaluate these rational functions, which may involve quite a complicated computation. 
\hfill$\triangle$
\end{remk}
\begin{remk}\label{rem3}
 Which linear subspaces may occur as the center subspace $W_c$ in a generic parameter-family of network dynamical systems, is a much more subtle question. The following partial answer is known.  A representation of a semigroup $\Sigma$ is called ``indecomposable'' if its representation space cannot be written as a non-trivial direct sum of invariant subspaces. A result known as the Krull-Schmidt theorem  states that every (finite dimensional) representation of $\Sigma$ can be written as the direct sum of indecomposable representations that is unique up to isomorphism. It is furthermore known that  an indecomposable representation can be classified as being of either real, complex or quaternionic type. It was shown in \cite{RinkSanders3} that under a specific condition on the representation of $\Sigma$, a one-parameter steady state bifurcation can generically only occur if the center subspace $W_c$ is an indecomposable representation of real type.  In particular, this is the case for the fundamental networks of our three  example networks {\bf A}, {\bf B} and {\bf C}. In these examples, the representation space splits as the direct sum of two indecomposable representations of real type, both of which may therefore occur as $W_c$, while the full space $V^3$ can generically not be equal to $W_c$.
 \hfill$\triangle$
\end{remk}

\noindent From the above discussion we see that the problem of finding generic bifurcations for homogenous coupled cell network vector fields is reduced to finding those for a class of equivariant reduced vector fields. As it turns out, this latter class admits a rather straightforward description which states that, roughly speaking, they come with a network structure themselves.  More precisely, we have the following theorem.

\begin{thr}\label{T6}
Let 
\begin{equation}\label{eq:5:21}
V^n = W_1 \oplus W_2 
\end{equation}
be a decomposition of the phase space of a fundamental network
into $\{A_{\sigma_i}\}$-invariant spaces, and denote by $P_1: V^n \rightarrow W_1$ and $i_{1} : W_1 \rightarrow V^n$ respectively the projection onto $W_1$ and the inclusion of $W_1$ into $V^n$. 

A map $F: W_1 \rightarrow W_1$ is $\{A_{\sigma_i}\}$-equivariant  if and only if the map $$i_{1}\circ F\circ P_{1}$$ is a fundamental network vector field.
\end{thr}

\begin{proof} Recall that both $P_1$ and $i_1$ are $\{A_{\sigma_i}\}$-equivariant maps. So if $i_{1}\circ F\circ P_{1}$ commutes with $A_{\sigma_i}$ for all $i$, then so does 
\begin{equation}\label{eq:5:22}
F = P_{1}\circ (i_{1}\circ F\circ P_{1})\circ i_{1} \, .
\end{equation}
For the same reason, $i_{1}\circ F\circ P_{1}$ is $\{A_{\sigma_i}\}$-equivariant when  $F$ is. Moreover, it follows from theorem \ref{symfun} that this property is equivalent to $i_{1}\circ F\circ P_{1}$ having the  structure of a fundamental network. Here we use in particular that $i_{1}\circ F\circ P_{1}$ is a vector field on $V^n$. This concludes the proof.
\end{proof}

\newpage
\section{Synchrony and the center manifold}\label{sec5}
Until now we have focused on developing a center manifold theory for fundamental networks. However, out of our three leading examples, only example {\bf A} is conjugate to its own fundamental network, while networks {\bf B} and {\bf C} are embedded in a fundamental network as a robust synchrony space. Moreover, by theorem \ref{inbed} this is true in general. The following theorem states that center manifold reduction respects robust synchrony spaces in a natural way.

\begin{thr}\label{T7}
Let  $\syn_P \subset V^n$ be a robust synchrony space  in a fundamental network. For every $\la_0 \in \Omega$, 
the map $P = (P_c,\Id): V^n \times \Omega \rightarrow W_c \times \Omega$ of theorem \ref{T5} maps the space
 \[ \{(x,\la) \in M_c: x \in  \syn_P, \la = \la_0  \}\]
bijectively onto the space 

\[\{(x,\la) \in W_c \times \Omega: x \in \syn_P, \la = \la_0  \}\, .\]
\end{thr}

\begin{proof}
Recall from theorem \ref{T5} that $P = P' \circ \ul{P}_c$ is an $\{\ul{A}_{\sigma_i} \}$-equivariant map that sends a vector $(x,\la) $ in $V^n \times \Omega$ to a vector in $W_c \times \Omega$ with the same $\la$-component. Therefore, keeping $\la = \la_0$ fixed we may think of $P$ as an $\{A_{\sigma_i} \}$-equivariant map from $V^n$ to $W_c$. Let us likewise use $M_c$ to denote what is really $\{(x,\la) \in M_c: \la = \la_0\}$. It follows from theorem \ref{T5} that, under these identifications, $P|_{M_c}: M_c \rightarrow W_c$ is an $\{A_{\sigma_i} \}$-equivariant bijection between $\{A_{\sigma_i}\}$-invariant sets. We will keep these identifications throughout this proof. In particular, what we want to show in this notation is that $P$ maps $M_c \cap \syn_P$ bijectively onto $W_c \cap \syn_P$.

For this purpose, let us denote by $i_c: W_c \rightarrow V^n$ the inclusion of $W_c$ into $V^n = W_c \oplus W_h$. The map $i_c \circ P$ is now an $\{A_{\sigma_i} \}$-equivariant map from $V^n$ into itself. Therefore, we may conclude by theorem \ref{symfun} that it is a fundamental network vector field. In particular, we see that it maps $\syn_P$ into itself and from this we conclude that $P|_{M_c}$ maps $M_c \cap \syn_P$ into $W_c \cap \syn_P$.

On the other hand, it follows that $(P|_{M_c})^{-1}: W_c \rightarrow M_c$ is $\{A_{\sigma_i}\}$-equivariant as well. Therefore, so is the function $i_{M_c} \circ (P|_{M_c})^{-1} \circ P_c: V^n \rightarrow V^n$, where we use $i_{M_c}:M_c \rightarrow V^n$ to denote the natural inclusion of $M_c$ into $V^n$. As before, we conclude that $i_{M_c} \circ (P|_{M_c})^{-1} \circ P_c$ is a fundamental network vector field and therefore sends $\syn_P$ into itself. From this it follows that $(P|_{M_c})^{-1}$ maps $W_c \cap \syn_P$ into $M_c \cap \syn_P$ and we conclude that this happens in fact bijectively. This proves the theorem.
\end{proof}
\noindent 
Recall that the linearisation of a fundamental network vector field gives rise to a decomposition of $V^n$ into invariant subspaces $W_c$ and $W_h$. As the possible dynamics on the former subspace is completely determined by the action of $\Sigma$ on this space, we may conclude that isomorphic splittings of $V^n$ into $W_c$ and $W_h$ give rise to conjugate dynamics and therefore equivalent bifurcations. However, this reasoning seems to lose sight of (robust) synchrony spaces, such as the one representing the original network vector field in its fundamental one. The following theorem settles this, as it tells us that synchrony spaces do behave well under choosing different decompositions of $V^n$ into invariant subspaces.

\begin{thr}\label{T8}
Let $\{W_i\}_{i=1}^k$ and $\{W'_i\}_{i=1}^k$ be two sets of $\{A_{\sigma_i}\}$-invariant subspaces of $V^n$ such that 
\begin{equation}
V^n = \bigoplus_{i=1}^k W_i = \bigoplus_{i=1}^k W'_i \, .
\end{equation}
Suppose furthermore that for every $i$, $W_i$ and $W'_i$ are isomorphic as $\{A_{\sigma_i}\}$-invariant subspaces. Then, for any robust synchrony space $\syn_P$ and any isomorphism $\phi_j:W_j \rightarrow W'_j$, it holds that $\phi_j$ restricts to a bijection between $\syn_P \cap W_j$ and $\syn_P \cap W'_j$. In particular, for every $j$ there exists an isomorphism between $W_j$ and $W'_j$ respecting $\syn_P$ in this way.
\end{thr}

\begin{proof}
It is clear that if we have proven that any isomorphism between $W_j$ and $W'_j$ respects $\syn_P$, that we have then shown that there exists an isomorphism respecting this synchrony space. This is because $W_j$ and $W'_j$ are isomorphic, i.e. there exists (at least one) isomorphism between them. Let $\phi_j$ now be an isomorphism between $W_j$ and $W'_j$. By choosing for every $i \not= j$  an isomorphism $\phi_i$ between $W_i$ and $W'_i$, we can define the function $\Phi:V^n \rightarrow V^n$ given by

\begin{equation}
\Phi: \sum_{i=1}^k x_i \mapsto \sum_{i=1}^k \phi_i(x_i) \, ,
\end{equation} 
for $x_i \in W_i$. First of all, because this map is an $\{A_{\sigma_i}\}$-equivariant map by construction, we conclude that it is in fact a fundamental network vector field. In particular, it sends $\syn_P$ to itself. Secondly, because it sends an element in $W_j$ to an element in $W'_j$, we conclude that $\Phi$ sends the space $\syn_P \cap W_j$ into the space $\syn_P \cap W'_j$. Lastly, because $\Phi|_{W_j} = \phi_j$ we conclude that $\phi_j$ sends $\syn_P \cap W_j$ into $\syn_P \cap W'_j$. By the same argument we see that $\phi_j^{-1}$ sends $\syn_P \cap W'_j$ into $\syn_P \cap W_j$, from which it follows that this happens in fact bijectively. This concludes the proof.\end{proof}

\begin{remk}
If we are given two decompositions  of $V^n$ into invariant subspaces 
\begin{equation}
V^n = W_c \oplus W_h= W'_c \oplus W'_h\quad
\end{equation}
and if we know that $W_c$ and $W'_c$ are isomorphic, then it follows that the same holds true for $W_h$ and $W'_h$.  Namely, writing $W_c$, $W'_c$, $W_h$ and $W'_h$ as the direct sum of indecomposable representations, we get two indecomposable splittings of $V^n$. By the Krull-Schmidt theorem such a splitting is unique, from which it follows that $W_h$ and $W'_h$ are indeed isomorphic as well. 
 \hfill$\triangle$
\end{remk}
\noindent We now have a recipe for classifying the generic bifurcations of a homogeneous coupled cell network. One has to go through the following steps:

\begin{itemize}
\item One  first constructs the fundamental network of the homogeneous network.
\item Next, one determines all possible representation types of generic center subspaces $W_c$ that can occur in a bifurcation.
\item After that, one determines all possible reduced vector fields of the fundamental network on $W_c$. This is equivalent  to finding all the equivariant vector fields on $W_c$. As it turns out, an efficient way of finding these is by using that $F: W_c \rightarrow W_c$ is symmetric if and only if $i_c \circ F \circ P_c :V^n \rightarrow V^n$ is a fundamental network vector field. See theorem \ref{T6}.
\item Finally, theorem \ref{T7} tells us that the dynamics on the center manifold of the original network  can be found by restricting the dynamics on the center manifold of the fundamental network to an appropriate synchrony space. Namely, we know that the dynamics of the original network vector field is embedded as a robust synchrony space inside the fundamental network and that center manifold reduction respects it.
\end{itemize}
Note that if one finds two decompositions $V^n  = W_c \oplus W_h   = W'_c \oplus W'_h$ such that $W_c$ and $W'_c$ are isomorphic as representations of $\Sigma$, then for any bifurcation that occurs along  $W_c$ there is an equivalent bifurcation along $W'_c$. By theorem \ref{T8} this equivalence respects robust synchrony spaces, and in particular the one that represents the original network.

\section{Examples}\label{sec6}

In this section, we illustrate the machinery that we have developed. We will show which co-dimension one steady state bifurcations one can expect in networks {\bf B} and {\bf C} when the phase space of a single cell is $V = \R$. In particular, it will become clear that the difference in generic bifurcations can be explained from the representations of the symmetry semigroups. For network {\bf A}, this was already shown in \cite{RinkSanders2} with the help of normal form theory.

\subsection{Network {\bf B}}\label{sec6B}
Recall from Section \ref{sec1} that network {\bf B} is realised as the robust synchrony space $\{X_2 = X_3\}$ inside the fundamental network 
\begin{equation} \label{vierfun2}
\begin{split}
\dot{X}_1 &= f(X_1,  X_{2}, X_3, X_4) \\
\dot{X}_2 &= f(X_2, X_{4}, X_3, X_4) \\
\dot{X}_3 &=f(X_3,  X_4, X_3, X_4)  \\
\dot{X}_4 &= f(X_4, X_{4}, X_3, X_4) 
\end{split}
\end{equation} 
where it can be found by setting $X_1 = x_1$, $X_2 = X_3 = x_2$ and $X_4 = x_3$. For the moment, we suppress the dependence of $f$ on the parameter $\la$ in our notation. Equation (\ref{vierfun2}) describes all vector fields on $\R^4$ that commute with the maps 
\begin{equation}
\begin{split}
&(X_1,X_2, X_3,X_4) \mapsto (X_2,X_4,X_3,X_4)\, ,\\
&(X_1,X_2, X_3,X_4) \mapsto (X_3,X_4,X_3,X_4)\, , \\
&(X_1,X_2, X_3,X_4) \mapsto (X_4 ,X_4,X_3,X_4)\, . 
\end{split}
\end{equation}
It can  be shown \cite{fibr} that any decomposition of $\R^4$ into indecomposable representations of these symmetries is isomorphic to the splitting
\begin{equation}\label{split2}
\R^4 = \{X_1 = X_2 = X_3 = X_4\} \oplus \{X_4 = 0\}
\end{equation}
with corresponding projections given by 
\begin{equation}
P(X_1, X_2, X_3,X_4) = (X_4,X_4,X_4,X_4)
\end{equation}
and
\begin{equation}
Q(X_1, X_2, X_3,X_4) = (X_1 - X_4,X_2 - X_4,X_3 - X_4,0)\, .
\end{equation}
Let us first assume that the center subspace is isomorphic to the  subrepresentation $\syn_0 := \{X_1 = X_2 = X_3 = X_4\}$. Theorem \ref{T6} says that a reduced vector field $F=F(X): \syn_0 \rightarrow \syn_0$ is equivariant  if and only if 
\begin{equation}\label{eq:5:26}
(i_{\syn_0} \circ F \circ P)(X_1, X_2, X_3, X_4) = (F(X_4), F(X_4), F(X_4), F(X_4))
\end{equation}
is a fundamental network vector field. As this is clearly the case, we see that there are no constraints on $F$. In particular, the bifurcation problem reduces to solving $F(X,\la) = 0$ given that $F(0,0) = 0$ and $D_{X}F(0,0) = 0$. This will generically yield a fully synchronous saddle node bifurcation.

Now for the representation $\{X_4 = 0\}$: if we parametrize it by $X_1$, $X_2$ and $X_3$, then a general vector field on this space can be written as 

\begin{equation}F(X_1,X_2,X_3) = \begin{pmatrix*}[c] F_1(X_1,X_2,X_3) \\ F_2(X_1,X_2,X_3)  \\ F_3(X_1,X_2,X_3)  \end{pmatrix*} \, . \end{equation} 
According to theorem \ref{T6}, the expression
\begin{equation}
i_{\{X_4 = 0\}} \circ F \circ Q(X_1, X_2, X_3, X_4) = \begin{pmatrix*}[c] F_1(X_1-X_4, X_2-X_4, X_3-X_4) \\ F_2(X_1-X_4, X_2-X_4, X_3-X_4) \\ F_3(X_1-X_4, X_2-X_4, X_3-X_4) \\ 0 \end{pmatrix*}
\end{equation}
must be a fundamental network vector field. Using that a fundamental network vector field is  determined by its first component, we obtain the equalities
\begin{equation}
{\footnotesize \begin{pmatrix*}[c] F_1(X_1-X_4, X_2-X_4, X_3-X_4) \\ F_2(X_1-X_4, X_2-X_4, X_3-X_4) \\ F_3(X_1-X_4, X_2-X_4, X_3-X_4) \\ 0 \end{pmatrix*} = \begin{pmatrix*}[l] F_1(X_1-X_4, X_2-X_4, X_3-X_4) \\ F_1(X_2-X_4,0, X_3-X_4) \\ F_1(X_3-X_4, 0, X_3-X_4) \\ F_1(0,0,X_3-X_4) \end{pmatrix*} \, . 
}
\end{equation}
Therefore, a general equivariant vector field on $\{X_4 = 0 \}$ is given by 

\begin{equation}F(X_1,X_2,X_3) = \begin{pmatrix*}[l] F_1(X_1,X_2,X_3) \\ F_1(X_2, 0 ,X_3)  \\ F_1(X_3, 0 ,X_3)  \end{pmatrix*} \, , \end{equation} with the additional condition that $F_1(0,0,X_3) = 0$. Since this means that we may set $F_1(X_1,X_2,X_3) = X_2G(X_1,X_2,X_3) + X_1H(X_1,X_3)$, we can write

\begin{equation}F(X_1,X_2,X_3)= \begin{pmatrix*}[r] X_2G(X_1,X_2,X_3) + X_1H(X_1,X_3) \\  X_2H(X_2,X_3)   \\X_3H(X_3,X_3)  \end{pmatrix*} \, . \end{equation} 
Recall also that we are only interested in the dynamics on the synchrony space $\{X_2 = X_3\}$. We thus have to solve the equations

\begin{equation} \label{red2}\begin{matrix*}[r] X_2G(X_1,X_2) + X_1H(X_1,X_2) = 0 \\  X_2H(X_2,X_2) = 0 \end{matrix*} \ . \end{equation} 
To solve these, let us include the parameter in our notation again and write 
\begin{equation}
G(X_1, X_2,\la) = C + \mathcal{O}(|X_1|+ |X_2|+ |\la|) \, ,
\end{equation}
and
\begin{equation}
H(X_1, X_2,\la) = a_1X_1 + a_2X_2 + a_{3}\la + \mathcal{O}(|X_1|^2+ |X_2|^2+ |\la|^2) \, .
\end{equation}
Note that $H(0,0,0) = 0$, which follows from the fact that the linearisation with respect to $X$ of the reduced vector field in (\ref{red2}) is noninvertible at the origin $X_1 = X_2 = \la = 0$. Focussing first on the second equation of (\ref{red2}), 
\begin{equation}
X_2H(X_2,X_2, \la) = X_2[(a_1+a_2)X_2+ a_{3}\la + \mathcal{O}( |X_2|^2+ |\la|^2)] = 0 \, ,
\end{equation}
we see that either $X_2 = 0$ or, if $a_1 + a_2 \not=0$, that $X_2 = X_2(\la) = -\frac{a_3}{a_1+a_2}\la + \mathcal{O}(|\la|^2)$ by the implicit function theorem. If we set $X_2 = 0$, then the first equation of (\ref{red2}) reduces to
\begin{equation}
X_1H(X_1,0, \la) = X_1[ a_1X_1+ a_{3}\la + \mathcal{O}(|X_1|^2 + |\la|^2)] = 0 \, .
\end{equation}
This either gives $X_1 = 0$ or $X_1 = X_1(\la) = -\frac{a_3}{a_1}\la + \mathcal{O}(|\la|^2)$ if $a_1 \not= 0$. If we set $X_2 = X_2(\la) = -\frac{a_3}{a_1+a_2}\la + \mathcal{O}(|\la|^2)$ then the first equation reduces to

\begin{equation}
-C\frac{a_3}{a_1+a_2}\la + a_1X_1^2 + \mathcal{O}(|X_1|^3 + |X_1||\la|+|\la|^2) = 0\, .
\end{equation}
Next, substituting $\la = \pm \mu^2$ and $X_1 = \mu Y$ gives us
\begin{equation}
\mp C\frac{a_3}{a_1+a_2}\mu^2 + a_1\mu^2Y^2 + \mathcal{O}(|\mu|^3) = 0\, ,
\end{equation}
or, after dividing by $\mu^2$,

\begin{equation}
\mp C\frac{a_3}{a_1+a_2} + a_1Y^2 + \mathcal{O}(|\mu|) = 0\, .
\end{equation}
For one choice of the sign in $\la = \pm \mu^2$ this gives no solutions with $\mu = 0$, whereas for the other we find the solutions
\begin{equation}
(Y,\mu) = \left( \pm \sqrt{\frac{|C a_3|}{|a_1(a_1 + a_2)|}},0 \right) \, .
\end{equation}
Assuming $C, a_3 \not=0$, the implicit function theorem now tells us that these solutions continue in $\mu$ as 

\begin{equation}
Y(\mu) = \pm \sqrt{\frac{|C a_3|}{|a_1(a_1 + a_2)|}} + \mathcal{O}(|\mu|) \, ,
\end{equation}
from which it follows that we have the branches 

\begin{equation}
X_1(\la) = \pm \sqrt{\frac{ C a_3}{a_1(a_1 + a_2)}\la} + \mathcal{O}(|\la|) \, .
\end{equation}
To summarise, we have found the following solutions to equation (\ref{red2}):
\begin{equation}
\begin{split}
&X_1(\la) = X_2(\la) = 0 \, , \\
&X_1(\la) = -\frac{a_3}{a_1}\la + \mathcal{O}(|\la|^2),  \ X_2(\la) = 0 \, , \\
&X_1(\la) = \pm \sqrt{\frac{ C a_3}{a_1(a_1 + a_2)}\la} + \mathcal{O}(|\la|), \ X_2(\la) = -\frac{a_3}{a_1+a_2}\la + \mathcal{O}(|\la|^2) \, .
\end{split}
\end{equation}
Note that in all cases we have $X_4 = 0$, hence the first branch is fully synchronous, the second is partially synchronous and the last is fully asynchronous.  

To determine the stability of these branches, we linearise the vector field in (\ref{red2}) in the $X$-variables to obtain the Jacobian
\begin{equation}\nonumber
{\footnotesize \begin{pmatrix*}[l]
2a_1X_1 + a_2X_2 + a_3\la + \mathcal{O}(|X_1|^2 + |X_2| + |\la|^2) & C +\mathcal{O}(|X_1| + |X_2| + |\la|) \\
0 & 2(a_1 + a_2)X_2 + a_3\la  +  \mathcal{O}(|X_2|^2 + |\la|^2)
\end{pmatrix*} \, .}
\end{equation}
For the fully synchronous  branch, this Jacobian reduces to 
\begin{center}
\begin{equation}
\begin{pmatrix*}[l]
a_3\la + \mathcal{O}( |\la|^2) & C +\mathcal{O}(|\la|) \\
0 &  a_3\la  +  \mathcal{O}( |\la|^2)
\end{pmatrix*} \, ,
\end{equation}
\end{center}
hence we find two times the eigenvalue $a_3\la  +  \mathcal{O}( |\la|^2)$. Likewise, for the partially synchronous branch we find $a_3\la  +  \mathcal{O}( |\la|^2)$ and $-a_3\la  +  \mathcal{O}( |\la|^2)$. For the fully nonsynchronous one we find $\pm 2 a_1 \sqrt{\frac{ C a_3}{a_1(a_1 + a_2)}\la}+ \mathcal{O}( |\la|)$ and $-a_3\la  +  \mathcal{O}( |\la|^2)$. In particular, we see that the partially synchronous branch is always a saddle, and the fully synchronous branch can only give its stability to the fully nonsynchronous one.

Recalling that network {\bf B} can be obtained from network $\widetilde{\bf B}$ by making the identifications $X_1 = x_1$, $X_2 = X_3 = x_2$ and $X_4 = x_3$, and using that the center subspace in $\widetilde{\bf B}$ is given by $\{ X_4 = 0\}$, the above analysis proves the claims  on network {\bf B} of the introduction.

\subsection{Network {\bf C}}\label{sec6C}

Recall from Section \ref{sec1} that network {\bf C} is realised inside the fundamental network 
\begin{equation} \label{fund3}
\begin{split}
\dot{X}_1 &= f(X_1,  X_{2}, X_3, X_4, X_5) \\
\dot{X}_2 &= f(X_2, X_{4}, X_3, X_4, X_5) \\
\dot{X}_3 &=f(X_3,  X_5, X_3, X_4, X_5)  \\
\dot{X}_4 &= f(X_4, X_{4}, X_3, X_4, X_5) \\
\dot{X}_4 &= f(X_5, X_{4}, X_3, X_4, X_5) 
\end{split}
\end{equation} 
by setting $X_1 = X_3 = x_1$, $X_2 = X_5 = x_2$ and $X_4 = x_3$. This latter system describes all vector fields on $\R^5$ with the symmetries 
\begin{equation}
\begin{split}
&(X_1,X_2,X_3,X_4,X_5) \mapsto (X_2, X_{4}, X_3, X_4, X_5)\, , \\
&(X_1,X_2,X_3,X_4,X_5) \mapsto (X_3,  X_5, X_3, X_4, X_5)\, , \\
&(X_1,X_2,X_3,X_4,X_5) \mapsto (X_4, X_{4}, X_3, X_4, X_5)\, ,  \\
&(X_1,X_2,X_3,X_4,X_5) \mapsto (X_5, X_{4}, X_3, X_4, X_5) \, . 
\end{split}
\end{equation}
As shown in \cite{fibr}, the center and hyperbolic subspaces of its linearisation at a fully synchronous point will  generically define a splitting of $\R^5$ isomorphic to 
\begin{equation}\label{split3}
\R^5 = \{X_1 = \dots = X_5\} \oplus \{X_4 = 0\} \, .
\end{equation}
Projections corresponding to this decomposition are given by 
\begin{equation}
P(X_1, X_2, X_3,X_4,X_5) = (X_4,X_4,X_4,X_4,X_4)
\end{equation}
and
\begin{equation}
Q(X_1, X_2, X_3,X_4,X_5) = (X_1 - X_4,X_2 - X_4, X_3 -X_4, 0, X_5- X_4) \, .
\end{equation}
If we take the fully synchronous space to be the center subspace, then generically we again obtain a fully synchronous saddle node bifurcation, as was the case in network {\bf B} as well. If instead we take $\{X_4 = 0\}$ to be the center subspace, then a reduced vector field for the system (\ref{fund3}) corresponds to an equivariant vector field on this space. Following theorem \ref{T6}, these correspond to the functions $F = (F_1, F_2, F_3, F_5): \R^4 \rightarrow \R^4$ such that the expression

\begin{equation}\nonumber
 i_{\{X_4 = 0\}} \circ F \circ Q(X_1,  \dots ,X_5) = \begin{pmatrix*}[c] F_1(X_1-X_4, X_2-X_4, X_3-X_4, X_5-X_4) \\ F_2(X_1-X_4, X_2-X_4, X_3-X_4, X_5-X_4)  \\ F_3(X_1-X_4, X_2-X_4, X_3-X_4, X_5-X_4) \\0 \\ F_5(X_1-X_4, X_2-X_4, X_3-X_4, X_5-X_4)  \end{pmatrix*}
\end{equation}
is a fundamental network vector field. This yields the equalities
\begin{equation}
\begin{split}
&\begin{pmatrix*}[c] F_1(X_1-X_4, X_2-X_4, X_3-X_4, X_5-X_4) \\ F_2(X_1-X_4, X_2-X_4, X_3-X_4, X_5-X_4)  \\ F_3(X_1-X_4, X_2-X_4, X_3-X_4, X_5-X_4) \\0 \\ F_5(X_1-X_4, X_2-X_4, X_3-X_4, X_5-X_4)   \end{pmatrix*}  =\\
& \begin{pmatrix*}[l] F_1(X_1-X_4, X_2-X_4, X_3-X_4, X_5-X_4) \\ F_1(X_2-X_4, 0, X_3-X_4, X_5-X_4)  \\ F_1(X_3-X_4, X_5-X_4, X_3-X_4, X_5-X_4) \\F_1(0, 0, X_3-X_4, X_5-X_4)  \\ F_1(X_5-X_4, 0, X_3-X_4, X_5-X_4)   \end{pmatrix*}  \, .
\end{split}
\end{equation}
It follows that a general  equivariant vector field on $\{X_4 = 0 \}$ is of the form
\begin{equation}
F(X_1, X_2, X_3, X_5)  = \begin{pmatrix*}[l] F_1(X_1, X_2, X_3, X_5) \\ F_1(X_2, 0, X_3, X_5)  \\ F_1(X_3, X_5, X_3, X_5) \\ F_1(X_5, 0, X_3, X_5)   \end{pmatrix*}  
\end{equation}
with the additional condition that $F_1(0, 0, X_3, X_5) = 0$. This latter condition can be reformulated by writing 
\begin{equation}
F_1(X_1,X_2,X_3,X_5) = X_1G(X_1,X_3,X_5)  + X_2H(X_1,X_2,X_3,X_5)\, ,
\end{equation}
from which it follows that a general $\Sigma$-equivariant vector field has the form
\begin{equation}
F(X_1, X_2, X_3, X_5)  = \begin{pmatrix*}[l] X_1G(X_1,X_3,X_5)  + X_2H(X_1,X_2,X_3,X_5)\\ X_2G(X_2,X_3,X_5)  \\ X_3G(X_3,X_3,X_5)  + X_5H(X_3,X_5,X_3,X_5) \\ X_5G(X_5,X_3,X_5)    \end{pmatrix*}  \ .
\end{equation}
If we now restrict to network {\bf C}, i.e. to the  synchrony space $\{X_1 = X_3, X_2 = X_5 \}$, then the steady state problem reduces to solving the equations 
\begin{equation} \label{3verg}
\begin{split}
X_1G(X_1,X_1,X_2,\la)  + X_2H(X_1,X_2,\la) &= 0\, , \\ X_2G(X_2,X_1,X_2,\la)  &= 0 \, .
\end{split}
\end{equation}
At this point, we include the parameter again to investigate  generic steady state bifurcations. So we shall write
\begin{equation}
\begin{split}
&G(X_1,X_2,X_3,\la) = a_1X_1 + a_2X_2 + a_3X_3 + a_4\la + \mathcal{O}(|X|^2 + |\la|^2)\, ,  \\
&H(X_1,X_2, \la) = C + b_1X_1 + b_2X_2 + b_3\la + \mathcal{O}(|X|^2 + |\la|^2) \, .
\end{split}
\end{equation}
The second line in equation (\ref{3verg}) is solved when $X_2 = 0$ or when
\begin{equation}\label{onder3}
G(X_2, X_1,X_2,\la) = a_1X_2 + a_2X_1 + a_3X_2 + a_4\la + \mathcal{O}(|X|^2 + |\la|^2) = 0 \, .
\end{equation}
Assuming $a_2 \not=0$, the implicit function theorem then gives us that locally all the solutions to equation (\ref{onder3}) are given by $X_1 = X_1(X_2, \la) = -\frac{a_1 + a_3}{a_2}X_2 - \frac{a_4}{a_2}\la + \mathcal{O}(|X_2|^2 + |\la|^2)$.

Let us first assume that $X_2 =0$. The first line in equation (\ref{3verg}) is then solved when $X_1 = 0$ or when $X_1(\la) = \frac{-a_4}{a_1 + a_2}\la + \mathcal{O}(|\la|^2)$, assuming $a_1 + a_2 \not= 0$. Next, suppose we have the relation $X_1 = X_1(X_2, \la) = -\frac{a_1 + a_3}{a_2}X_2 - \frac{a_4}{a_2}\la +  \mathcal{O}(|X_2|^2 + |\la|^2)$. The first line in (\ref{3verg}) then becomes the equation 
\begin{equation}\nonumber 
\begin{split}
&\left[ -\frac{a_1 + a_3}{a_2}X_2 - \frac{a_4}{a_2}\la\right] \left((a_1 + a_2)\left[ -\frac{a_1 + a_3}{a_2}X_2 - \frac{a_4}{a_2}\la \right] + a_3X_2 + a_4\la \right) + \\
&X_2 \left(   C + b_1\left[-\frac{a_1 + a_3}{a_2}X_2 - \frac{a_4}{a_2}\la \right] + b_2X_2 + b_3\la  \right) + \mathcal{O}(|X_2|^3 + |\la|^3)= 0 \, ,
\end{split}
\end{equation}
which can be rewritten as

\begin{equation}
CX_2 + \frac{a_4^2 a_1}{a_2^2} \la^2+ \mathcal{O}(|X_2|^2 + |\la||X_2|+ |\la|^3) = 0\, .
\end{equation}
Hence, assuming $C \not= 0$, the implicit function theorem gives the solution
\begin{equation}
X_2 = X_2(\la) = \frac{-a_4^2 a_1}{C a_2^2 }\la^2 + \mathcal{O}(|\la|^3) \, .
\end{equation}
Combined with the  relation $X_1 = X_1(X_2, \la) = -\frac{a_1 + a_3}{a_2}X_2 - \frac{a_4}{a_2}\la + \mathcal{O}(|X_2|^2 + |\la|^2)$, we then get 
\begin{equation}
 X_1( \la) =   \frac{-a_4}{a_2}\la + \mathcal{O}( |\la|^2)\, .
\end{equation}
To summarise, we have found the three bifurcation branches
\begin{equation}\label{br3}
\begin{split}
&X_1(\la) = X_2(\la) = 0\, ,  \\
&X_1(\la) = \frac{-a_4}{a_1 + a_2}\la + \mathcal{O}(|\la|^2), \ X_2(\la) = 0\, ,\\
&X_1( \la) =   \frac{-a_4}{a_2}\la + \mathcal{O}( |\la|^2), \ X_2(\la) = \frac{-a_4^2 a_1}{C a_2^2 }\la^2 + \mathcal{O}(|\la|^3) \, ,
\end{split}
\end{equation}
where furthermore we have that $X_4 = 0$ in all three cases. Note that this makes the first branch fully synchronous, the second partially synchronous and the last fully non synchronous. Note however that this third branch is partially synchronous up to first order. 

A stability analysis similar to that in Section \ref{sec6B}  yields the eigenvalue $a_4\la + \mathcal{O}(|\la|^2)$ twice for the fully synchronous branch. We thus assume that $a_4\neq 0$. For the partially synchronous branch we then find the eigenvalues $-a_4\la + \mathcal{O}(|\la|^2)$ and $\frac{a_1a_4}{a_1+a_2}\la + \mathcal{O}(|\la|^2)$. For the fully non-synchronous branch we find $\beta_1\la + \mathcal{O}(|\la|^2)$ and $\beta_2\la + \mathcal{O}(|\la|^2)$, where $\beta_1$ and $\beta_2$ satisfy 
\begin{equation}
\beta_1 + \beta_2 = -a_4 \frac{2a_1 + a_2}{a_2} \quad \text{and} \quad \beta_1 \cdot \beta_2 = a_4^2\frac{a_1}{a_2} \, .
\end{equation}
Note that for positive values of $\frac{a_1}{a_2}$ the expression $\frac{2a_1 + a_2}{a_2} = 2\frac{a_1}{a_2} + 1$ is necessarily positive as well. Hence, the fully non-synchronous branch either takes over the stability of the fully synchronous one, or remains a saddle. The same holds true for the partially synchronous solution. However, when this latter branch gains the stability of the fully synchronous one then it must hold that $\frac{a_1}{a_1+a_2} < 0$. From this it follows that $\frac{a_2}{a_1} = \frac{a_1+a_2}{a_1} -1 < 0$ and we see that in this case the nonsynchronous branch is necessarily a saddle. We note that it is also possible that both the partially synchronous and the fully nonsynchronous branch are saddles, as there are values of $a_1$ and $a_2$ for which $\frac{a_1}{a_2}$ is negative, but $\frac{a_1}{a_1+a_2} = \left(\frac{a_2}{a_1} + 1\right)^{-1}$ is positive. 

As network {\bf C} is obtained from  the fundamental system (\ref{fund3}) by setting $X_1 = X_3 = x_1$, $X_2 = X_5 = x_2$ and $X_4 = x_3$, we see that the results obtained above hold for this former system under these identifications. This proves the claims on network {\bf C} of the introduction. 

  \bibliography{CoupledNetworks}
\bibliographystyle{amsplain}

\end{document}